\documentclass{amsart}

\title{The monoidal center and the character algebra}
\author{Kenichi Shimizu}
\date{}

\address{Department of Mathematical Sciences, Shibaura Institute of Technology, 307
Fukasaku, Minuma-ku, Saitama-shi, Saitama 337-8570, Japan.}
\email{kshimizu@shibaura-it.ac.jp}

\usepackage{amsmath,amssymb,amscd,bbm,stmaryrd}
\usepackage[all,knot]{xy}
\usepackage{url}
\usepackage{textgreek}

\usepackage[numbers,sort&compress]{natbib}

\usepackage{amsthm}
\numberwithin{equation}{section}

\newtheorem{counter}{}[section]

\theoremstyle{definition}
\newtheorem{definition}         [counter]{Definition}

\newtheorem*{notation*}         {Notation}
\theoremstyle{plain}
\newtheorem{lemma}              [counter]{Lemma}

\newtheorem{proposition}        [counter]{Proposition}
\newtheorem{theorem}            [counter]{Theorem}
\newtheorem{corollary}          [counter]{Corollary}

\newtheorem*{theorem*}          {Theorem}
\theoremstyle{remark}
\newtheorem{remark}             [counter]{Remark}
\newtheorem{example}            [counter]{Example}

\newcommand{\id}{\mathrm{id}}
\newcommand{\eval}{{\rm ev}}
\newcommand{\coev}{{\rm coev}}
\newcommand{\op}{\mathsf{op}}
\newcommand{\rev}{\mathsf{rev}}
\newcommand{\mir}{\mathsf{mir}}

\newcommand{\unitobj}{\mathbbm{1}}

\newcommand{\Hom}{\mathrm{Hom}}
\newcommand{\End}{\mathrm{End}}

\newcommand{\Nat}{\mathrm{Nat}}

\newcommand{\LEX}{\mathrm{Lex}}
\newcommand{\Trace}{\mathrm{Tr}}
\newcommand{\trace}{\mathrm{tr}}

\newcommand{\Kdelta}{\mathop{\text{\textdelta}}\nolimits}

\newcommand{\Gr}{\mathsf{Gr}}
\newcommand{\CE}{\mathsf{CE}}
\newcommand{\CF}{\mathsf{CF}}

\begin{document}

\begin{abstract}
  For a pivotal finite tensor category $\mathcal{C}$ over an algebraically closed field $k$,
  we define the algebra $\mathsf{CF}(\mathcal{C})$ of class functions and the internal character $\mathsf{ch}(X) \in \mathsf{CF}(\mathcal{C})$ for an object $X \in \mathcal{C}$ by using an adjunction between $\mathcal{C}$ and its monoidal center $\mathcal{Z}(\mathcal{C})$.
  We also develop the theory of integrals and the Fourier transform in a unimodular finite tensor category by using the same adjunction.
  Our main result is that the map $\mathsf{ch}: \mathsf{Gr}_k(\mathcal{C}) \to \mathsf{CF}(\mathcal{C})$ given by taking the internal character is a well-defined injective homomorphism of $k$-algebras,
  where $\mathsf{Gr}_k(\mathcal{C})$ is the scalar extension of the Grothendieck ring of $\mathcal{C}$ to $k$.
  Moreover, under the assumption that $\mathcal{C}$ is unimodular, the map $\mathsf{ch}$ is an isomorphism if and only if $\mathcal{C}$ is semisimple.

  As an application, we show that the algebra $\mathsf{Gr}_{k}(\mathcal{C})$ is semisimple if $\mathcal{C}$ is a non-degenerate pivotal fusion category.
  If, moreover, $\mathsf{Gr}_k(\mathcal{C})$ is commutative, then we define the character table of $\mathcal{C}$ based on the integral theory.
  It turns out that the character table is obtained from the $S$-matrix if $\mathcal{C}$ is a modular tensor category.
  Generalizing corresponding results in the finite group theory, we prove the orthogonality relations and the integrality of the character table.
\end{abstract}

\maketitle

\section{Introduction}

The character theory is an important subject in the theory of groups and Hopf algebras.
Since many results on Hopf algebras have been generalized and understood in the setting of tensor categories,
it is interesting to find a category-theoretical counterpart of the character theory.
The aim of this paper is to propose and develop such a framework.

To explain what we do in this paper, we first consider a representation $V$ of a finite group $G$ over a field $k$.
The character $\chi_V: G \to k$ is defined by taking the trace of the action of an element of $G$ on $V$.
The basic property of the trace implies that $\chi_V$ is in fact a class function on $G$.
In representation-theoretic terms, this fact can be said that the linear map $\chi_V: k G \to k$ is a homomorphism of $k G$-modules if we view $A := k G$ and $k$ as the adjoint representation and the trivial representation, respectively.
From the viewpoint of the theory of tensor categories, the trivial representation is the unit object.
A counterpart of the adjoint representation may not be obvious at the first glance.
To see what it should be, we remark that the adjoint representation of a Hopf algebra $H$ arises from an adjunction between the category of $H$-modules and the category of Yetter-Drinfeld $H$-modules.
This observation suggests that a category-theoretical counterpart of the adjoint representation can be defined in terms of the monoidal center.

A {\em finite tensor category} \cite{MR2119143} is a class of tensor categories including the representation category of a finite-dimensional Hopf algebra.
In this paper, we initiate the ``internal character theory'' based on the above idea.
More precisely, let $\mathcal{C}$ be a finite tensor category over an algebraically closed field $k$,
and let $\mathcal{Z}(\mathcal{C})$ be its monoidal center.
Then the forgetful functor $U: \mathcal{Z}(\mathcal{C}) \to \mathcal{C}$ has a right adjoint, say $R$.
We define the {\em adjoint algebra} $A \in \mathcal{C}$ by $A = U R(\unitobj)$, where $\unitobj \in \mathcal{C}$ is the unit object.
It turns out that the adjoint algebra is a generalization of the adjoint representation of a Hopf algebra (which is in fact a module algebra over the Hopf algebra).
Thus we call
\begin{equation*}
  \CE(\mathcal{C}) := \Hom_{\mathcal{C}}(\unitobj, A)
  \quad \text{and} \quad
  \CF(\mathcal{C}) := \Hom_{\mathcal{C}}(A, \unitobj)
\end{equation*}
the {\em space of central elements} and the {\em space of class functions}, respectively.
A certain universal property of the adjoint algebra yields an isomorphism
\begin{equation}
  \label{eq:intro-CE-iso}
  \CE(\mathcal{C}) \cong \End(\id_{\mathcal{C}})
\end{equation}
of vector spaces. By the definition of $R$, we also have an isomorphism
\begin{equation}
  \label{eq:intro-CF-iso}
  \CF(\mathcal{C})
  = \Hom_{\mathcal{C}}(U R(\unitobj), \unitobj)
  \cong \Hom_{\mathcal{Z}(\mathcal{C})}(R(\unitobj), R(\unitobj))
  = \End_{\mathcal{Z}(\mathcal{C})}(R(\unitobj))
\end{equation}
of vector spaces. Thus we can introduce algebra structures on $\CE(\mathcal{C})$ and $\CF(\mathcal{C})$ so that \eqref{eq:intro-CE-iso} and~\eqref{eq:intro-CF-iso} are isomorphisms of algebras.

Like the adjoint representation of a Hopf algebra, the algebra $A$ acts on every object of $\mathcal{C}$.
Thus, if $\mathcal{C}$ has a pivotal structure, we can define the {\em internal character} $\mathsf{ch}(X): A \to \unitobj$ of $X \in \mathcal{C}$ to be the partial pivotal trace of the action $A \otimes X \to X$ (this definition basically agrees with that given in \cite{MR3146014} when $\mathcal{C}$ has a ribbon structure; see Remark~\ref{rem:coend-Hopf}).
Now let $\Gr(\mathcal{C})$ be the Grothendieck ring of $\mathcal{C}$,
and put $\Gr_k(\mathcal{C}) := k \otimes_{\mathbb{Z}} \Gr(\mathcal{C})$.
Our first main result implies that the linear map
\begin{equation}
  \label{eq:intro-ch-map}
  \mathsf{ch}: \Gr_k(\mathcal{C}) \to \CF(\mathcal{C}),
  \quad [X] \mapsto \mathsf{ch}(X)
  \quad (X \in \mathcal{C})
\end{equation}
is a well-defined injective algebra map (Theorem~\ref{thm:irr-ch-indep} and its corollary).
This map is not surjective in general.
Our second main result is that, under the assumption that $\mathcal{C}$ is {\em unimodular} in the sense of \cite{MR2097289}, there holds:
\begin{equation}
  \label{eq:intro-ch-unimo}
  \text{$\mathsf{ch}: \Gr_k(\mathcal{C}) \to \CF(\mathcal{C})$ is an isomorphism}
  \iff
  \text{$\mathcal{C}$ is semisimple}
\end{equation}
(Theorem~\ref{thm:unimo-gr}).
Thus, if $\mathcal{C}$ is not semisimple, there is an element of $\CF(\mathcal{C})$ that cannot be expressed as a linear combination of internal characters.
This result seems to suggest the importance of the assumption of unimodularity in the study of finite tensor categories,
taking into account the fact that \eqref{eq:intro-ch-unimo} does not hold in general without the unimodularity.

Our proof of these results is based on the fact that $\mathcal{Z}(\mathcal{C})$ can be identified with the category of modules over a certain Hopf monad on $\mathcal{C}$ \cite{MR2342829,MR2869176}.
For the proof, we also set up the integral theory and define the Fourier transform for a unimodular finite tensor category.
These techniques are of independent interest as generalizations of corresponding results on Hopf algebras given in \cite{MR1265853,MR2349620,MR2464107,MR2863455}.

A {\em fusion category} \cite{MR2183279} is a semisimple finite tensor category and one of well-studied classes of tensor categories.
In this paper, we give some applications of our results to the theory of fusion categories.
If $\mathcal{C}$ is a pivotal fusion category, then the map \eqref{eq:intro-ch-map} is an isomorphism.
Hence, by~\eqref{eq:intro-CF-iso}, we have an isomorphism
\begin{equation}
  \label{eq:intro-Gr-C-iso}
  \Gr_k(\mathcal{C}) \cong \End_{\mathcal{Z}(\mathcal{C})}(R(\unitobj))
\end{equation}
of algebras. This result allows us to study the Grothendieck algebra of a fusion category via the monoidal center. For example, we prove a conjecture of Ostrik \cite{2015arXiv150301492O} stating that
$\Gr_k(\mathcal{C})$ is semisimple if and only if $\mathcal{C}$ is non-degenerate in the sense of \cite{MR2183279} (Theorem~\ref{thm:Gro-alg-ss}). See below for other results.

The existence of an isomorphism such as \eqref{eq:intro-Gr-C-iso} has been showed by Ostrik \cite{2013arXiv1309.4822O} in the case where $k$ is of characteristic zero.
Our construction of the isomorphism is somewhat more canonical and works as well in the case of positive characteristic.
We also note that Neshveyev and Yamashita \cite{2015arXiv150107390N} have established such an isomorphism for $C^*$-tensor categories possibly with infinitely many isomorphism classes of simple objects.
Instead of mentioning such $C^*$-tensor categories,
we concern applications to finite tensor categories and fusion categories over a field of arbitrary characteristic.
Our approach is essentially same as \cite{2015arXiv150107390N} if it is limited to $C^*$-fusion categories.

\subsection*{Organization of this paper}

This paper is organized as follows:
In Section~\ref{sec:preliminaries}, we recall some basic results on monoidal categories that will be used throughout this paper.

In Section~\ref{sec:internal-ch}, we first recall the construction of the central Hopf (co)monad on a rigid monoidal category, which plays a crucial role in this paper.
Provided that certain (co)ends exists in a rigid monoidal category $\mathcal{C}$, a right adjoint of the forgetful functor $U: \mathcal{Z}(\mathcal{C}) \to \mathcal{C}$ is expressed by the central Hopf comonad $\bar{Z}$ on $\mathcal{C}$.
We establish isomorphisms~\eqref{eq:intro-CE-iso} and~\eqref{eq:intro-CF-iso} in terms of the central Hopf comonad.
Finally, we consider the case where $\mathcal{C}$ is the representation category of a Hopf algebra
and explain in detail what our results say (such an additional explanation is provided each of Sections~\ref{sec:indep} and \ref{sec:integral}).

In Section~\ref{sec:indep}, we prove the injectivity of the map \eqref{eq:intro-ch-map} for a pivotal finite tensor category $\mathcal{C}$.
For the proof, we require a certain algebra $A^{\mathrm{ss}} \in \mathcal{C}$, which corresponds to the quotient of a finite-dimensional Hopf algebra $H$ by the Jacobson radical if $\mathcal{C}$ is the representation category of $H$.
We show that there is an epimorphism $q: A \to A^{\mathrm{ss}}$ and every internal character factors through $q$.
Then the proof of the injectivity goes along the same line as the proof of the linear independence of irreducible characters of a finite-dimensional Hopf algebra.

Techniques used in Section~\ref{sec:indep} seem to be essential when we deal with certain form of coends. With motivation coming from the recent study of conformal field theories \cite{MR3289342,MR3146014}, in Section~\ref{sec:indep}, we also compute the composition factors of such a coend by applying our techniques (Theorem~\ref{thm:comp-factor-coends}).

In Section~\ref{sec:integral}, we introduce the notions of an integral, a cointegral and the Fourier transformation in a unimodular finite tensor category by using a characterization of the unimodularity in terms of the monoidal center \cite{2014arXiv1402.3482S}. By utilizing these tools, we prove not only the assertion \eqref{eq:intro-ch-unimo}, but also a Maschke-type theorem and a Radford-type trace formula.

In Section~\ref{sec:appl-fusion}, we give applications of our results to fusion categories.
We first prove the above-mentioned conjecture of Ostrik.
We also show that, for a non-degenerate pivotal fusion category $\mathcal{C}$, the algebra $\Gr_k(\mathcal{C})$ is commutative if and only if $R(\unitobj)$ is multiplicity-free (Theorem~\ref{thm:Gro-alg-comm}).
Finally, by using the Fourier transform, we define the character table and related notions of a pivotal fusion category $\mathcal{C}$ such that $\Gr_k(\mathcal{C})$ is commutative.
Our definitions generalize Cohen and Westreich's \cite{MR2863455}.
We prove some results on the character table, such as the orthogonality relations and the cyclotomic integrality.
If $\mathcal{C}$ is a modular tensor category, then its character table is obtained from the $S$-matrix (Example~\ref{ex:ch-t-MTC}).

\subsection*{Acknowledgments}

The author thanks Victor Ostrik and Makoto Yamashita for letting him know about \cite{2013arXiv1309.4822O} and \cite{2015arXiv150107390N}, respectively, and giving him some valuable comments.
The author also thanks Takahiro Hayashi for valuable comments and suggestions.
The author appreciates the referee for careful reading of the manuscript.
The author is currently supported by JSPS KAKENHI Grant Number JP16K17568.

\section{Preliminaries}
\label{sec:preliminaries}

\subsection{Monoidal categories}

For the basic theory of monoidal categories, we refer the reader to \cite{MR1712872} and \cite{MR1321145}. In this paper, all monoidal categories are assumed to be strict. For a monoidal category $\mathcal{C} = (\mathcal{C}, \otimes, \unitobj)$ with tensor product $\otimes$ and unit object $\unitobj$, we set $\mathcal{C}^{\op} = (\mathcal{C}^{\op}, \otimes, \unitobj)$ and $\mathcal{C}^{\rev} = (\mathcal{C}, \otimes^{\rev}, \unitobj)$, where $\otimes^{\rev}$ is the reversed tensor product given by $X \otimes^{\rev} Y = Y \otimes X$.

A {\em monoidal functor} \cite[XI.2]{MR1712872} is a functor $F: \mathcal{C} \to \mathcal{D}$ between monoidal categories $\mathcal{C}$ and $\mathcal{D}$
endowed with a morphism $F_0: \unitobj \to F(\unitobj)$ in $\mathcal{C}$ and a natural transformation
$F_2(X, Y): F(X) \otimes F(Y) \to F(X \otimes Y)$ ($X, Y \in \mathcal{C}$)
satisfying a certain coherence condition. We say that a monoidal functor $(F, F_2, F_0)$ is {\em strong} if $F_0$ and $F_2$ are invertible, and {\em strict} if they are identities.
A {\em comonoidal functor} from $\mathcal{C}$ to $\mathcal{D}$ is the same thing as a monoidal functor from $\mathcal{C}^{\op}$ to $\mathcal{D}^{\op}$.

A {\em left dual object} of $X \in \mathcal{C}$ is a triple $(X', e, d)$ consisting of an object $X' \in \mathcal{C}$ and morphisms $e: X' \otimes X \to \unitobj$ and $d: \unitobj \to X \otimes X'$ such that
\begin{equation*}
  (e \otimes \id_{X'}) \circ (\id_{X'} \otimes d) = \id_{X'}
  \quad \text{and} \quad
  (\id_{X} \otimes e) \circ (d \otimes \id_{X}) = \id_{X}.
\end{equation*}
Suppose that every object of $\mathcal{C}$ has a left dual object (we say that $\mathcal{C}$ is {\em left rigid} if this is the case). For each $X \in \mathcal{C}$, we choose a left dual object
\begin{equation*}
  (X^*, \ \eval_X: X^* \otimes X \to \unitobj, \ \coev_X: \unitobj \to X \otimes X^*)
\end{equation*}
of $X \in \mathcal{C}$. Then the assignment $X \mapsto X^*$ gives rise to a strong monoidal functor $(-)^*: \mathcal{C} \to \mathcal{C}^{\op,\rev}$, which we call the {\em left duality functor}.

A {\em rigid monoidal category} is a monoidal category $\mathcal{C}$ such that both $\mathcal{C}$ and $\mathcal{C}^{\rev}$ are left rigid. If $\mathcal{C}$ is a rigid monoidal category, then the left duality functor $(-)^*$ of $\mathcal{C}$ is an equivalence. A quasi-inverse of $(-)^*$, denoted by ${}^*(-)$, is called the {\em right duality functor}. As explained in \cite[Lemma 5.4]{2013arXiv1309.4539S}, we may assume that $(-)^*$ and ${}^*(-)$ are strict monoidal and mutually inverse to each other. Thus,
\begin{equation*}
  (X \otimes Y)^* = Y^* \otimes X^*,
  \quad \unitobj^* = \unitobj,
  \quad \text{and} \quad ({}^*X)^* = X = {}^*(X^*).
\end{equation*}

\subsection{Pivotal monoidal category}

A {\em pivotal structure} of a rigid monoidal category $\mathcal{C}$ is an isomorphism $j: \id_{\mathcal{C}} \to (-)^{**}$ of monoidal functors.
A {\em pivotal monoidal category} is a rigid monoidal category endowed with a pivotal structure.
Now let $\mathcal{C}$ be a pivotal monoidal category with pivotal structure $j$.
Throughout this paper, we use the following notation:
\begin{equation}
  \label{eq:eval-tilde-def}
  \widetilde{\eval}_X := \eval_{X^*} \circ (j_X \otimes \id_{X^*})
  \quad (X \in \mathcal{C}).
\end{equation}
Let $A, B, X \in \mathcal{C}$ be objects. For a morphism $f: A \otimes X \to B \otimes X$, the (right) {\em partial pivotal trace} of $f$ is defined and denoted by
\begin{equation*}
  \trace_{A,B}^X(f) = (\id_B \otimes \widetilde{\eval}_{X}) \circ (f \otimes \id_{X^*}) \circ (\id_A \otimes \coev_X)
  \in \Hom_{\mathcal{C}}(A, B).
\end{equation*}
If $A = B = \unitobj$, then we write $\trace_{A,B}^X(f) \in \End_{\mathcal{C}}(\unitobj)$ simply as $\trace(f)$ and call it the (right) {\em pivotal trace} of $f$. The (right) {\em pivotal dimension} of $X \in \mathcal{C}$ is defined by
\begin{equation*}
  \dim(X) := \trace(\id_X).
\end{equation*}

\subsection{The monoidal center}

For a monoidal category $\mathcal{C}$, the {\em monoidal center} (or the {\em Drinfeld center}) of $\mathcal{C}$ is a category $\mathcal{Z}(\mathcal{C})$ defined as follows:
An object of $\mathcal{Z}(\mathcal{C})$ is a pair $(V, \sigma)$ consisting of an object $V \in \mathcal{C}$ and a natural isomorphism
\begin{equation}
  \label{eq:left-half-br}
  \sigma_X: V \otimes X \to X \otimes V \quad (X \in \mathcal{C})
\end{equation}
satisfying a part of the hexagon axiom.
A morphism $f: (V, \sigma) \to (V', \sigma')$ in $\mathcal{Z}(\mathcal{C})$ is a morphism $f: V \to V'$ in $\mathcal{C}$ such that
$(\id_X \otimes f) \circ \sigma_X = \sigma'_X \circ (f \otimes \id_X)$
for all $X \in \mathcal{C}$.
The composition is defined in an obvious way.

A {\em braiding} \cite[XIII.1]{MR1321145} of a monoidal category $(\mathcal{B}, \otimes_{\mathcal{B}}, \unitobj)$ is a natural isomorphism $\sigma: \otimes_{\mathcal{B}}^{} \to \otimes_{\mathcal{B}}^{\rev}$ satisfying the hexagon axiom,
and a {\em braided monoidal category} is a monoidal category endowed with a braiding.
The category $\mathcal{Z}(\mathcal{C})$ is in fact a braided monoidal category; see, {\it e.g.}, \cite[XIII.3]{MR1321145} for details.

\begin{remark}
  Replacing \eqref{eq:left-half-br} with ``$\sigma_X: X \otimes V \to V \otimes X$ ($X \in \mathcal{C}$)'', one can define another version of the monoidal center, which we denote by $\mathcal{Z}_r(\mathcal{C})$.
  The braided monoidal category introduced in \cite[XIII.3]{MR1321145} under the name ``center'' is in fact $\mathcal{Z}_r(\mathcal{C})$.
  Given a braided monoidal category $\mathcal{B} = (\mathcal{B}, \otimes, \unitobj, \sigma)$, we set
  \begin{equation*}
    \mathcal{B}^{\rev} = (\mathcal{B}, \otimes^{\rev}, \unitobj, \sigma^{\rev})
    \quad \text{and} \quad
    \mathcal{B}^{\mir} = (\mathcal{B}, \otimes, \unitobj, \sigma^{\mir}),
  \end{equation*}
  where $\sigma^{\rev}_{X,Y} = \sigma_{Y,X}^{}$ and $\sigma^{\mir}_{X,Y} = \sigma_{Y,X}^{-1}$. Then
  we have $\mathcal{Z}(\mathcal{C}) = \mathcal{Z}_r(\mathcal{C}^{\rev})^{\rev,\mir}$.
\end{remark}

\subsection{Hopf monads}

Let $\mathcal{C}$ be a monoidal category. A {\em bimonad} \cite{MR2355605} on $\mathcal{C}$ is a monad $(T, \mu, \eta)$ on $\mathcal{C}$ endowed with a comonoidal structure
\begin{equation*}
  T_0: T(\unitobj) \to \unitobj
  \quad \text{and} \quad
  T_2(V, W): T(V \otimes W) \to T(V) \otimes T(W)
  \quad (V, W \in \mathcal{C})
\end{equation*}
such that the multiplication $\mu: T^2 \to T$ and the unit $\eta: \id_{\mathcal{C}} \to T$ of the monad $T$ are comonoidal natural transformations. If $T$ is a bimonad on $\mathcal{C}$, then the category ${}_T \mathcal{C}$ of $T$-modules ($=$ the Eilenberg-Moore category of $T$ \cite[VI.2]{MR1712872}) is a monoidal category in such way that the forgetful functor $U: {}_T \mathcal{C} \to \mathcal{C}$ is a strong monoidal functor.

Now suppose that $\mathcal{C}$ is a rigid monoidal category. Then a {\em Hopf monad} \cite{MR2355605} on $\mathcal{C}$ is a bimonad $T$ on $\mathcal{C}$ endowed with natural transformations
\begin{equation*}
  S_V^{(\ell)}: T(T(V)^*) \to V^*
  \quad \text{and} \quad
  S_V^{(r)}: T({}^*T(V)) \to {}^*V
  \quad (V \in \mathcal{C}),
\end{equation*}
called the left antipode and the right antipode, respectively, satisfying certain conditions. If $T$ is a Hopf monad on $\mathcal{C}$, then the monoidal category ${}_T \mathcal{C}$ is rigid.

\subsection{Ends and coends}

Let $\mathcal{C}$ and $\mathcal{V}$ be categories, and let $S, T: \mathcal{C}^{\op} \times \mathcal{C} \to \mathcal{V}$ be functors. A {\em dinatural transformation} $\xi: S \xrightarrow{..} T$ is a family
\begin{equation*}
  \xi = \{ \xi_{X}: S(X,X) \to T(X,X) \}_{X \in \mathcal{C}}
\end{equation*}
of morphisms in $\mathcal{V}$ such that
\begin{equation*}
  T(\id_X, f) \circ \xi_X \circ S(f, \id_X) = T(f, \id_Y) \circ \xi_Y \circ S(\id_Y, f)
\end{equation*}
for all morphism $f: X \to Y$ in $\mathcal{C}$. An {\em end} of $S$ is a pair $(E, p)$ consisting of an object $E \in \mathcal{V}$ (regarded as a constant functor from $\mathcal{C}^{\op} \times \mathcal{C}$ to $\mathcal{V}$) and a dinatural transformation $p: E \xrightarrow{..} S$ that enjoys the following universal property:
For any pair $(E', p')$ consisting of an object $E' \in \mathcal{V}$ and a dinatural transformation $p': E' \xrightarrow{..} S$, there exists a unique morphism $\phi: E' \to E$ in $\mathcal{V}$ such that $p'_X = p_X \circ \phi$ for all $X \in \mathcal{C}$. A {\em coend} of $T$ is a pair $(C, i)$ consisting of an object $C \in \mathcal{V}$ and a dinatural transformation $i: T \xrightarrow{..} C$ having a similar universal property. The end of $S$ and the coend of $T$ are expressed as
\begin{equation*}
  \int_{X \in \mathcal{C}} S(X,X)
  \quad \text{and} \quad
  \int^{X \in \mathcal{C}} T(X,X),
\end{equation*}
respectively. See \cite[IX]{MR1712872} for the basic results on (co)ends.

\begin{lemma}
  \label{lem:coend-adj-1}
  Let $\mathcal{B}$, $\mathcal{C}$ and $\mathcal{V}$ be categories, and let $F: \mathcal{B} \to \mathcal{C}$ be a functor that has a right adjoint $G$. Then, for any functor $\Phi: \mathcal{C}^{\op} \times \mathcal{C} \to \mathcal{V}$, we have
  \begin{equation}
    \label{eq:coend-iso-1}
    \int^{X \in \mathcal{C}} \Phi(X, F G(X)) \cong \int^{V \in \mathcal{B}} \Phi(F(V), F(V)),
  \end{equation}
  meaning that if either one of coends exists, then the other one exists and there is a canonical isomorphism between them.
\end{lemma}

This is a special case of \cite[Lemma 3.9]{MR2869176}. We recall how to define \eqref{eq:coend-iso-1} from its proof.
Let $C_1$ and $C_2$ denote the left hand side and the right hand side of \eqref{eq:coend-iso-1}, respectively. The isomorphism $\phi_{1 2}: C_1 \to C_2$ mentioned in the above lemma is a unique morphism such that
\begin{equation*}
  \phi_{1 2} \circ i^{(1)}_X = i^{(2)}_{G(X)} \circ \Phi(\varepsilon_{X}, \id_{F G(X)})
\end{equation*}
for all $X \in \mathcal{C}$, where $i^{(a)}$ ($a = 1, 2$) is the universal dinatural transformation to $C_a$ and $\varepsilon: F G \to \id_{\mathcal{C}}$ is the counit of the adjunction.

We now assume that the coend $C_3 = \int^{X \in \mathcal{C}} \Phi(X, X)$ also exists. By the universal property, there exists a unique morphism $\phi_{2 3}: C_2 \to C_3$ such that
\begin{equation*}
  \phi_{2 3} \circ i_V^{(2)} = i_{F(V)}^{(3)}
\end{equation*}
for all $V \in \mathcal{B}$, where $i^{(3)}$ is the universal dinatural transformation to $C_3$. Thus
\begin{align*}
  \phi_{2 3} \circ \phi_{1 2} \circ i_{X}^{(1)}
  & = \phi_{2 3} \circ i^{(2)}_{G(X)} \circ \Phi(\varepsilon_{X}, \id_{F G(X)}) \\
  & = i_{F G(X)}^{(3)} \circ \Phi(\varepsilon_{X}, \id_{F G(X)}) \\
  & = i_{X}^{(3)} \circ \Phi(\id_X, \varepsilon_{X})
\end{align*}
for all $X \in \mathcal{C}$, where the third equality follows from the dinaturality of $i^{(3)}$. Summarizing the above argument, we have the following conclusion:

\begin{lemma}
  \label{lem:coend-adj-2}
  With the above notation, the composition
  \begin{equation*}
    \int^{X \in \mathcal{C}} \Phi(X, F G(X))
    \xrightarrow{\quad \phi_{1 2} \quad} \int^{V \in \mathcal{B}} \Phi(F(V), F(V))
    \xrightarrow{\quad \phi_{2 3} \quad} \int^{X \in \mathcal{C}} \Phi(X, X)
  \end{equation*}
  coincides with the morphism
  \begin{equation*}
    \int^{X \in \mathcal{C}} \Phi(\id_X, \varepsilon_X):
    \int^{X \in \mathcal{C}} \Phi(X, F G(X))
    \to \int^{X \in \mathcal{C}} \Phi(X, X).
  \end{equation*}
\end{lemma}

\subsection{Conjugation by the duality functor}
\label{subsec:adj-rig-mon}

For a functor $T: \mathcal{B} \to \mathcal{C}$ between rigid monoidal categories, we define $T^!: \mathcal{B} \to \mathcal{C}$ by
\begin{equation*}
  T^!(V) = T({}^* V)^*
  \quad (V \in \mathcal{B}).
\end{equation*}
A natural transformation $\alpha: S \to T$ between $S, T: \mathcal{B} \to \mathcal{C}$ induces
\begin{equation*}
  \alpha^!_V:
  T^!(V) = T({}^* V)^*  \xrightarrow{\quad (\alpha_{{}^*V})^* \quad} S({}^*V)^* = S^!(V)
  \quad (V \in \mathcal{B}).
\end{equation*}
The operation $(-)^!$ preserves the composition of functors between rigid monoidal categories and the horizontal composition of natural transformations, but reverses the vertical composition of natural transformations. Moreover, this operation is invertible: The inverse is given by $T \mapsto {}^{!}T$, where ${}^{!}T(V) = {}^*T(V^*)$. By the above observation, one easily proves:
\begin{enumerate}
\item $F: \mathcal{B} \to \mathcal{C}$ is a monoidal functor if and only if $F^!$ is comonoidal.
\item $T: \mathcal{C} \to \mathcal{C}$ is a monad on $\mathcal{C}$ if and only if $T^!$ is a comonad on $\mathcal{C}$.
  Moreover, if $T$ is a monad, then there is a category isomorphism between the categories of $T$-modules and of $T^!$-comodules.
\item Let $F: \mathcal{B} \to \mathcal{C}$ and $G: \mathcal{C} \to \mathcal{B}$ be functors. Then $F \dashv G$ ({\it i.e.}, $F$ is left adjoint to $G$) if and only if $G^! \dashv F^!$.
\end{enumerate}
Now suppose that $F: \mathcal{B} \to \mathcal{C}$ is a strong monoidal functor.
Since $F$ commutes with the duality functor, we have $F^! \cong F$.
Thus, for functors $L, R: \mathcal{C} \to \mathcal{B}$, we have
\begin{equation*}
  L \dashv F \iff F \dashv L^!
  \quad \text{and} \quad
  F \dashv R \iff R^! \dashv F.
\end{equation*}

\begin{remark}
  Some of the above results are found in \cite{MR2355605} and \cite{MR2869176}.
\end{remark}

\section{Internal characters}
\label{sec:internal-ch}

\subsection{The central Hopf monad}

Throughout this section, we let $\mathcal{C}$ be a rigid monoidal category such that the coend
\begin{equation}
  \label{eq:Hopf-monad-Z}
  Z(V) = \int^{X \in \mathcal{C}} X^* \otimes V \otimes X
\end{equation}
exists for all $V \in \mathcal{C}$. Day and Street \cite{MR2342829} showed that the functor $V \mapsto Z(V)$ has a structure of a monad such that the category ${}_Z\mathcal{C}$ of $Z$-modules is canonically isomorphic to the monoidal center $\mathcal{Z}(\mathcal{C})$. By Tannaka reconstruction, the monad $Z$ has a structure of a quasitriangular Hopf monad in the sense of \cite{MR2355605,MR2869176} such that the category isomorphism ${}_Z\mathcal{C} \cong \mathcal{Z}(\mathcal{C})$ is in fact an isomorphism of braided monoidal categories.

We first recall the definition of the Hopf monad $Z$ and the construction of the isomorphism ${}_Z\mathcal{C} \cong \mathcal{Z}(\mathcal{C})$. Let $i_{V;X}: X^* \otimes V \otimes X \to Z(V)$ be the universal dinatural transformation for the coend~\eqref{eq:Hopf-monad-Z}. The comonoidal structure
\begin{equation*}
  Z_0: Z(\unitobj) \to \unitobj
  \quad \text{and} \quad
  Z_2(V, W): Z(V \otimes W) \to Z(V) \otimes Z(W)
  \quad (V, W \in \mathcal{C})
\end{equation*}
of $Z$ is determined by $Z_{0} \circ i_{\unitobj; X} = \eval_X$ and
\begin{equation}
  \label{eq:Hopf-monad-como}
  Z_{2}(V, W) \circ i_{X; V \otimes W}
  = (i_{X;V} \otimes i_{X;W}) \circ (\id_{X^*} \otimes \id_{V} \otimes \coev_X \otimes \id_{W} \otimes \id_{X})
\end{equation}
for all $V, W, X \in \mathcal{C}$. The multiplication of $Z$ is determined by
\begin{equation}
  \label{eq:Hopf-monad-Z-mu}
  \mu_V \circ i^{(2)}_{V; X, Y} = i_{V; X \otimes Y}
\end{equation}
for all $V, X, Y \in \mathcal{C}$, where $i^{(2)}_{V; X, Y} = i_{Z(V); Y} \circ (\id_{Y^*} \otimes i_{V; X} \otimes \id_Y)$.
Finally, we define the unit $\eta: \id_{\mathcal{C}} \to Z$ of $Z$ by $\eta_V = i_{V; \unitobj}$ for $V \in \mathcal{C}$. We omit the description of the left antipode, the right antipode and the universal R-matrix of $Z$ since we will not use them in this paper. See \cite{MR2869176} for details, where, more generally, the quantum double of a Hopf monad is considered.

\begin{definition}
  We call $Z$ the {\em central Hopf monad} on $\mathcal{C}$.
\end{definition}

To establish the isomorphism ${}_Z \mathcal{C} \cong \mathcal{Z}(\mathcal{C})$, we define $\partial_{V,X}$ to be the morphism corresponding to $i_{V;X}$ under the canonical bijection
\begin{equation*}
  \Hom_{\mathcal{C}}(X^* \otimes V \otimes X, Z(V)) \cong \Hom_{\mathcal{C}}(V \otimes X, X \otimes Z(V)).
\end{equation*}
Given a $Z$-module $(M, a)$ with action $a: Z(M) \to M$, we define
\begin{equation*}
  \sigma_X: M \otimes X
  \xrightarrow{\quad \partial_{M,X} \quad}
  X \otimes Z(M)
  \xrightarrow{\quad \id_X \otimes a \quad}
  X \otimes M
  \quad (X \in \mathcal{C}).
\end{equation*}
The assignment $(M, a) \mapsto (M, \sigma)$ gives rise to an isomorphism ${}_Z \mathcal{C} \cong \mathcal{Z}(\mathcal{C})$ of braided monoidal categories.

\subsection{The adjoint algebra}
\label{subsec:adj-obj}

As we have seen, the monoidal center $\mathcal{Z}(\mathcal{C})$ can be identified with the category of modules over the central Hopf monad $Z$ on $\mathcal{C}$. Under the identification, the free $Z$-module functor
\begin{equation}
  \label{eq:Hopf-monad-Z-free}
  L: \mathcal{C} \to \mathcal{Z}(\mathcal{C}),
  \quad V \mapsto (Z(V), \mu_V)
  \quad (V \in \mathcal{C})
\end{equation}
is a left adjoint of the forgetful functor $U: \mathcal{Z}(\mathcal{C}) \to \mathcal{C}$. By the argument of \S\ref{subsec:adj-rig-mon}, the functor $U$ also has a right adjoint, say $R$, which is isomorphic to $L^!$.

By a {\em Hopf comonad} on $\mathcal{C}$, we mean a comonad $T$ on $\mathcal{C}$ endowed with a monoidal structure such that, in a word, $T^{\op}: \mathcal{C}^{\op} \to \mathcal{C}^{\op}$ is a Hopf monad on $\mathcal{C}^{\op}$.
Again by the argument of \S\ref{subsec:adj-rig-mon}, the functor $\bar{Z} := U R$ ($\cong Z^!$) has a structure of a Hopf comonad whose category of comodules is identical to ${}_Z \mathcal{C}$.

\begin{definition}
  \label{def:central-Hopf-comonad}
  We call $\bar{Z}$ and $A := \bar{Z}(\unitobj)$ the {\em central Hopf comonad} on $\mathcal{C}$ and the {\em adjoint algebra} in $\mathcal{C}$, respectively.
\end{definition}

The central Hopf comonad $\bar{Z}$ is more convenient for our applications than the central Hopf monad $Z$.
The reason why we have introduced $Z$ is just a technical one that some useful references, such as \cite{MR2342829,MR2355605}, deal with $Z$.

For simplicity, we assume $R = L^!$ (thus $\bar{Z} = Z^!$).
Then the universal dinatural transformation $i_{V;X}$ for the coend $Z(V)$ induces a morphism
\begin{equation}
  \label{eq:Hopf-comonad-Z-dinat}
  \pi_{V; X}: \bar{Z}(V) \xrightarrow{\quad (i_{{}^* \! V; \, {}^* \! X})^* \quad}
  ({}^*X^* \otimes {}^*V \otimes {}^*X)^* = X \otimes V \otimes X^*
\end{equation}
in $\mathcal{C}$ that is natural in $V \in \mathcal{C}$ and dinatural in $X \in \mathcal{V}$. Since the duality functor is an anti-equivalence, we have
\begin{equation}
  \label{eq:Hopf-comonad-Z}
  \bar{Z}(V) = \int_{X \in \mathcal{C}} X \otimes V \otimes X^*
\end{equation}
with the universal dinatural natural transformation given by~\eqref{eq:Hopf-comonad-Z-dinat}. Thus there is a natural bijection
\begin{equation}
  \label{eq:Hopf-comonad-Z-represent}
  \Hom_{\mathcal{C}}(V, \bar{Z}(W))
  \cong \Nat(V \otimes (-), (-) \otimes W),
\end{equation}
where $\Nat(F, G)$ is the set of natural transformations from $F$ to $G$.

The Hopf comonad structure of $\bar{Z}$ can be described in terms of $\pi$. By~\eqref{eq:Hopf-monad-Z-mu}, the comultiplication $\delta: \bar{Z} \to \bar{Z}^2$ is the unique natural transformation such that
\begin{equation}
  \label{eq:Hopf-comonad-Z-comult}
  (\id_X \otimes \pi_{V; Y} \otimes \id_{X^*}) \circ \pi_{\bar{Z}(V); X} \circ \delta_{V}
  = \pi_{V; X \otimes Y}
\end{equation}
for all $V, X, Y \in \mathcal{C}$. The counit is given by $\varepsilon_{V} = \pi_{V; \unitobj}$. We omit to describe the monoidal structure of $\bar{Z}$, but note that the multiplication $m$ and the unit $u$ of the adjoint algebra $A = \bar{Z}(\unitobj)$ are determined by
\begin{gather}
  \label{eq:adj-obj-unit}
  \pi_{\unitobj; X} \circ u = \coev_X, \\
  \label{eq:adj-obj-mult}
  \pi_{\unitobj; X} \circ m = (\id_X \otimes \eval_X \otimes \id_{X^*}) \circ (\pi_{\unitobj; X} \otimes \pi_{\unitobj; X})
\end{gather}
for all $X \in \mathcal{C}$, respectively.

\begin{definition}
  \label{def:can-act}
  We define the {\em canonical action} $\rho_X: A \otimes X \to X$ ($X \in \mathcal{C}$) to be the natural transformation corresponding to $\id_A$ via the bijection \eqref{eq:Hopf-comonad-Z-represent} with $V = A$ and $W = \unitobj$. More precisely, $\rho_X$ is given by the composition
  \begin{equation}
    \label{eq:can-act-by-pi}
    \rho_X: A \otimes X
    \xrightarrow{\quad \pi_{\unitobj; X} \otimes \id_X \quad}
    X \otimes X^* \otimes X
    \xrightarrow{\quad \id_X \otimes \eval_X \quad}
    X.
  \end{equation}
  By~\eqref{eq:adj-obj-unit} and~\eqref{eq:adj-obj-mult}, we see that $(X, \rho_X)$ is a left $A$-module in $\mathcal{C}$.
\end{definition}

\begin{remark}
  \label{rem:can-act-half-br}
  By definition, $A$ has a natural isomorphism $\sigma: A \otimes (-) \to (-) \otimes A$ such that $(A, \sigma)$ is an object of $\mathcal{Z}(\mathcal{C})$. Explicitly, it is given by
  \begin{equation*}
    \sigma_X = (\id_X \otimes \id_A \otimes \eval_X) \circ (\pi_{\bar{Z}(\unitobj); X} \otimes \id_X) \circ (\delta_{\unitobj} \otimes \id_X)
  \end{equation*}
  for $X \in \mathcal{C}$. By~\eqref{eq:Hopf-comonad-Z-comult}, the canonical action is expressed by $\sigma$ as follows:
  \begin{equation}
    \label{eq:can-act-by-half-br}
    \rho_X = (\id_X \otimes \varepsilon_{\unitobj}) \circ \sigma_X \quad (X \in \mathcal{C}).
  \end{equation}
  We have used \eqref{eq:can-act-by-pi} to define $\rho$, since it is convenient when we discuss the internal characters introduced in later.
  On the other hand, \eqref{eq:can-act-by-half-br} has an advantage that it does not involve the universal dinatural transformation $\pi$.
  Thus, by using~\eqref{eq:can-act-by-half-br}, we can define $A$ and $\rho$ whenever $U: \mathcal{Z}(\mathcal{C}) \to \mathcal{C}$ has a right adjoint even if $\mathcal{C}$ is not rigid (such as the case where $\mathcal{C}$ is the category of all left modules over a Hopf algebra with bijective antipode). In this paper, we do not discuss further generalizations in such a direction.
\end{remark}

\subsection{The monoid of central elements}
\label{subsec:CE}

Till the end of this section, $U$, $R$, $\bar{Z}$, $A$, $m$ and $u$ have the same meaning as in the previous subsection. We now consider the following set defined in terms of the adjoint algebra:

\begin{definition}
  $\CE(\mathcal{C}) := \Hom_{\mathcal{C}}(\unitobj, A)$ is called the set of {\em central elements}.
\end{definition}

Specializing \eqref{eq:Hopf-comonad-Z-represent} to $V = W = \unitobj$, we get a bijection
\begin{equation}
  \label{eq:CE-and-End-id}
  \psi: \CE(\mathcal{C}) \to \End(\id_{\mathcal{C}}),
  \quad \psi(a)_X = \rho_X(a \otimes \id_X)
  \quad (a \in \CE(\mathcal{C}), X \in \mathcal{C})
\end{equation}
({\it cf}. \cite[Proposition 5.2.5]{MR1862634}).
For $a, b \in \CE(\mathcal{C})$, we set $a \cdot b := m \circ (a \otimes b)$. Then the set $\CE(\mathcal{C})$ is a monoid with respect to this operation.
Moreover, the bijection \eqref{eq:CE-and-End-id} is in fact an isomorphism of monoids.

The following operator on $\CE(\mathcal{C})$ will be used in later:

\begin{definition}
  \label{def:antipodal}
  The {\em antipodal operator} is the operator $\mathfrak{S}$ on $\CE(\mathcal{C})$ induced by
  \begin{equation*}
    (-)^!: \End(\id_{\mathcal{C}}) \to \End(\id_{\mathcal{C}}),
    \quad \xi \mapsto \xi^!
  \end{equation*}
  via the bijection~\eqref{eq:CE-and-End-id}. See \S\ref{subsec:adj-rig-mon} for the definition of $\xi^!$.
\end{definition}

\subsection{The monoid of class functions}
\label{subsec:CF}

As we will see in later, the adjoint algebra generalizes the adjoint representation of a group.
Thus it is natural to consider the following set:

\begin{definition}
  $\CF(\mathcal{C}) := \Hom_{\mathcal{C}}(A, \unitobj)$ is called the set of {\em class functions}.
\end{definition}

For class functions $f$ and $g$, we define their product $f \star g \in \CF(\mathcal{C})$ by
\begin{equation}
  \label{eq:class-ft-prod}
  f \star g = f \circ \bar{Z}(g) \circ \delta_{\unitobj},
\end{equation}
where $\delta: \bar{Z} \to \bar{Z}{}^2$ is the comultiplication of $\bar{Z}$. This operation is nothing but the composition of morphisms in the co-Kleisli category of the comonad $\bar{Z}$. Thus $\CF(\mathcal{C})$ is a monoid with respect to $\star$.

The category of $\bar{Z}$-comodules can be identified with the category of $Z$-modules, and hence with $\mathcal{Z}(\mathcal{C})$. Since the co-Kleisli category of $\bar{Z}$ is equivalent to the category of free $\bar{Z}$-comodules, we have the following theorem:

\begin{theorem}
  \label{thm:class-ft-induc}
  The adjunction isomorphism
  \begin{equation}
    \label{eq:class-ft-adj-iso}
    \CF(\mathcal{C})
    = \Hom_{\mathcal{C}}(U R(\unitobj), \unitobj)
    \cong \End_{\mathcal{Z}(\mathcal{C})}(R(\unitobj)),
    \quad f \mapsto \bar{Z}(f) \circ \delta_{\unitobj}
  \end{equation}
  is an isomorphism of monoids.
\end{theorem}

\subsection{Internal characters}

We now suppose that the monoidal category $\mathcal{C}$ has a pivotal structure $j: \id_{\mathcal{C}} \to (-)^{**}$. We introduce a category-theoretical analogue of the notion of the character of a representation:

\begin{definition}
  The {\em internal character} of $X \in \mathcal{C}$ is the class function
  \begin{equation*}
    \mathsf{ch}(X) := \trace_{A, \unitobj}^X(\rho_X) \in \CF(\mathcal{C}).
  \end{equation*}
\end{definition}

By the definition of the canonical action, we have
\begin{equation}
  \label{eq:character-pi}
  \mathsf{ch}(X) = \widetilde{\eval}_X \circ \pi_{\unitobj; X},
\end{equation}
where $\widetilde{\eval}_X: X \otimes X^* \to \unitobj$ is the morphism defined by~\eqref{eq:eval-tilde-def}. Hence,
\begin{equation}
  \label{eq:character-unit-obj}
  \mathsf{ch}(\unitobj)
  = \widetilde{\eval}_{\unitobj} \circ \pi_{\unitobj; \unitobj}
  = \pi_{\unitobj; \unitobj}
  = \varepsilon_{\unitobj}.
\end{equation}

\begin{theorem}
  \label{thm:ch-mult}
  $\mathsf{ch}(X \otimes Y) = \mathsf{ch}(X) \star \mathsf{ch}(Y)$ for all $X, Y \in \mathcal{C}$.
\end{theorem}
\begin{proof}
  Since $j_{X \otimes Y} = j_X \otimes j_Y$, we have
  \begin{equation}
    \label{eq:eval-tilde}
    \widetilde{\eval}_{X \otimes Y} = \widetilde{\eval}_X \circ (\id_X \otimes \widetilde{\eval}_Y \otimes \id_{X^*})
  \end{equation}
  for all $X, Y \in \mathcal{C}$. Hence,
  \begin{align*}
    \allowdisplaybreaks
    & \mathsf{ch}(X) \star \mathsf{ch}(Y) \\
    & = \widetilde{\eval}_X \circ \pi_{\unitobj; X} \circ \bar{Z}(\mathsf{ch}(Y)) \circ \delta_{\unitobj}
    & & \text{(by~\eqref{eq:class-ft-prod} and~\eqref{eq:character-pi})} \\
    & = \widetilde{\eval}_X \circ (\id_X \otimes \mathsf{ch}(Y) \otimes \id_{X^*}) \circ \pi_{\bar{Z}(\unitobj); X} \circ \delta_{\unitobj}
    & & \text{(by the naturality of $\pi$)} \\
    & = \widetilde{\eval}_X \circ (\id_X \otimes \widetilde{\eval}_Y \otimes \id_{X^*}) \\
    & \qquad \qquad \circ (\id_X \otimes \pi_{\unitobj; Y} \otimes \id_{X^*}) \circ \pi_{\bar{Z}(\unitobj); X} \circ \delta_{\unitobj}
    & & \text{(by \eqref{eq:character-pi})} \\
    & = \widetilde{\eval}_{X \otimes Y} \circ \pi_{\unitobj; X \otimes Y}
    & & \text{(by~\eqref{eq:Hopf-comonad-Z-comult} and \eqref{eq:eval-tilde})} \\
    & = \mathsf{ch}(X \otimes Y)
    & & \text{(by \eqref{eq:character-pi})}. \qedhere
  \end{align*}
\end{proof}

\begin{remark}
  The graphical calculus for the central Hopf monad $Z$, explained in \cite{2013arXiv1309.4539S}, would be helpful to understand the above computation.
\end{remark}

\begin{remark}
  \label{rem:coend-Hopf}
  The coend $F = \int^{X \in \mathcal{C}} X^* \otimes X$ has a structure of a coalgebra in $\mathcal{C}$ and every object of $\mathcal{C}$ has a natural structure of a right $F$-comodule, which we call the canonical coaction (see \cite[\S5.2.2]{MR1862634}). Taking the left partial pivotal trace of the canonical coaction, we obtain a morphism $\mathsf{ch}'(X): \unitobj \to F$.
  The morphism $\mathsf{ch}'(X)$ has been considered in the study of conformal field theories \cite{MR3289342,MR3146014}.
  If we realize the adjoint algebra $A$ by $A = F^*$, then we have
  \begin{equation*}
    \mathsf{ch}'(X) = {}^*\mathsf{ch}(X^*)
  \end{equation*}
  for all $X \in \mathcal{C}$.
  Thus the theory of $\mathsf{ch}$ and that of $\mathsf{ch}'$ are essentially same.
  We prefer to deal with $\mathsf{ch}(X)$ since it is more convenient for our applications.
\end{remark}

\subsection{The braided case}

Suppose that $\mathcal{C}$ has a braiding $\sigma$.
Then the coend $F$ in Remark~\ref{rem:coend-Hopf} is not only a coalgebra but a Hopf algebra in $\mathcal{C}$ \cite{MR1381692,MR1862634}.
Hence the adjoint algebra is also a Hopf algebra as the dual of $F$.
The comultiplication $\underline{\Delta}$ of $A$ is a unique morphism such that
\begin{equation}
  \label{eq:transmutation-comul}
  (\pi_{\unitobj; X} \otimes \pi_{\unitobj; Y}) \circ \underline{\Delta}
  = (\id_X \otimes \sigma_{Y \otimes Y^*, X^*}) \circ \pi_{\unitobj; X \otimes Y}
\end{equation}
for all $X, Y \in \mathcal{C}$, and the counit of $A$ is $\varepsilon_{\unitobj}$.
We omit the description of the antipode since it will not be used.

As we have mentioned in Remark~\ref{rem:coend-Hopf}, a construction of internal characters has been appeared in \cite{MR3289342,MR3146014}.
The product of characters is defined by using the multiplication of the Hopf algebra $F$ in these papers.
Translating its definition into our context, we introduce a binary operation $\tilde{\star}$ on $\CF(\mathcal{C})$ by
$f \mathop{\tilde{\star}} g := (f \otimes g) \circ \underline{\Delta}$.
Now there are two binary operations on $\CF(\mathcal{C})$. These operations coincide:

\begin{proposition}
  $f \star g = f \mathop{\tilde{\star}} g$ for all $f, g \in \CF(\mathcal{C})$.
\end{proposition}
\begin{proof}
  For $V \in \mathcal{C}$, there is a unique morphism $\mathcal{H}_V^{(\ell)}: \bar{Z}(V) \to A \otimes V$ such that
  \begin{equation}
    \label{eq:comonad-Z-left-Hopf-oper}
    (\pi_{\unitobj; X} \otimes \id_V) \circ \mathcal{H}_V^{(\ell)}
    = (\id_X \otimes \sigma_{V,X^*}) \circ \pi_{V; X}
  \end{equation}
  for all $X \in \mathcal{C}$. By the naturality of $\sigma$ and $\pi$, we see that $\mathcal{H}^{(\ell)}_V$ is natural in the variable $V$. Let $\delta: \bar{Z} \to \bar{Z}{}^2$ be the comultiplication of $\bar{Z}$. Then we compute
  \begin{align*}
    \allowdisplaybreaks
    & (\pi_{\unitobj; X} \otimes \pi_{\unitobj; Y}) \circ \mathcal{H}_A^{(\ell)} \circ \delta_{\unitobj} \\
    & = (\id_X \otimes \id_{X^*} \otimes \pi_{\unitobj; Y}) \circ (\id_X \otimes \sigma_{A,X^*}) \circ \pi_{A; X} \circ \delta_{\unitobj}
    & & \text{(by \eqref{eq:comonad-Z-left-Hopf-oper})} \\
    & = (\id_X \otimes \sigma_{Y \otimes Y^*, X}) \circ (\id_X \otimes \pi_{\unitobj;Y} \otimes \id_{X^*}) \circ \pi_{A; X} \circ \delta_{\unitobj}
    & & \text{(by the naturality)} \\
    & = (\id_X \otimes \sigma_{Y \otimes Y^*, X}) \circ \pi_{\unitobj; X \otimes Y}
    & & \text{(by \eqref{eq:Hopf-comonad-Z-comult})} \\
    & = (\pi_{\unitobj; X} \otimes \pi_{\unitobj; Y}) \circ \underline{\Delta}
    & & \text{(by~\eqref{eq:transmutation-comul})}
  \end{align*}
  for all $X, Y \in \mathcal{C}$. Hence we obtain
  \begin{equation}
    \label{eq:transmutation-comul-Hopf-l}
    \underline{\Delta} = \mathcal{H}_A^{(\ell)} \circ \delta_{\unitobj}.
  \end{equation}
  Using this formula, we compute
  \begin{align*}
    \allowdisplaybreaks
    f \mathop{\tilde{\star}} g
    & = (f \otimes g) \circ \mathcal{H}_A^{(\ell)} \circ \delta_{\unitobj} \\
    & = f \circ \mathcal{H}^{(\ell)}_{\unitobj} \circ \bar{Z}(g) \circ \delta_{\unitobj}
    & & \text{(by the naturality of $\mathcal{H}^{(\ell)}$)} \\
    & = f \circ\bar{Z}(g) \circ \delta_{\unitobj}
    & & \text{(since $\mathcal{H}^{(\ell)}_{\unitobj}$ is the identity)} \\
    & = f \star g. & & \qedhere
  \end{align*}
\end{proof}

\begin{theorem}
  If $\mathcal{C}$ is braided, then the monoid $\CF(\mathcal{C})$ is commutative.
\end{theorem}
\begin{proof}
  We define $\mathcal{H}_V^{(r)}: \bar{Z}(V) \to V \otimes A$ by
  \begin{equation}
    \label{eq:comonad-Z-right-Hopf-oper}
    (\id_V \otimes \pi_{\unitobj; X}) \circ \mathcal{H}_V^{(r)}
    = (\sigma_{X,V} \otimes \id_{X^*}) \circ \pi_{V; X}
  \end{equation}
  for $X \in \mathcal{C}$ ({\it cf}. \eqref{eq:comonad-Z-left-Hopf-oper}). By the definition of a braiding, we have
  \begin{equation*}
    \sigma_{Y \otimes Y^*, X^*}
    = (\sigma_{Y,X^*} \otimes \id_{Y^*}) \circ (\id_Y \otimes \sigma_{Y^*,X^*})
    = (\sigma_{Y,X^*} \otimes \id_{Y^*}) \circ (\id_Y \otimes \sigma_{Y,X}^*)
  \end{equation*}
  for all $X, Y \in \mathcal{C}$ (see \cite[XIII--XIV]{MR1321145}). Hence,
  \begin{align*}
    & (\pi_{\unitobj; X} \otimes \pi_{\unitobj; Y}) \circ \underline{\Delta} \\
    & = (\id_X \otimes \sigma_{Y,X^*} \otimes \id_{Y^*}) \circ (\id_X \otimes \id_Y \otimes \sigma_{Y,X}^*) \circ \pi_{\unitobj; X \otimes Y}
    & & \text{(by~\eqref{eq:transmutation-comul})} \\
    & = (\id_X \otimes \sigma_{Y,X^*} \otimes \id_{Y^*}) \circ (\sigma_{Y,X} \otimes \id_{X^*} \otimes \id_{Y^*}) \circ \pi_{\unitobj; Y \otimes X}
    & & \text{(by the dinaturality)} \\
    & = (\sigma_{Y, X \otimes X^*} \otimes \id_{Y^*}) \circ \pi_{\unitobj; Y \otimes X}
  \end{align*}
  for all $X, Y \in \mathcal{C}$.
  We obtain $\underline{\Delta} = \mathcal{H}^{(r)}_{A} \circ \delta_{\unitobj}$ in a similar way as~\eqref{eq:transmutation-comul-Hopf-l} in the proof of the previous lemma but by using this expression instead of \eqref{eq:transmutation-comul}. Thus,
  \begin{align*}
    g \star f
    & = (g \otimes f) \circ \underline{\Delta} \\
    & = (g \otimes f) \circ \mathcal{H}^{(r)}_A \circ \delta_{\unitobj} \\
    & = f \circ \mathcal{H}^{(r)}_{\unitobj} \circ \bar{Z}(g) \circ \delta_{\unitobj}
    & & \text{(by the naturality of $\mathcal{H}^{(r)}$)} \\
    & = f \circ \bar{Z}(g) \circ \delta_{\unitobj}
    & & \text{(since $\mathcal{H}^{(r)}_{\unitobj}$ is the identity)} \\
    & = f \star g
  \end{align*}
  for all class functions $f$ and $g$.
\end{proof}

\subsection{The case of Hopf algebras}
\label{subsec:Hopf-alg}

Let $H$ be a finite-dimensional Hopf algebra over a field $k$ with comultiplication $\Delta$, counit $\varepsilon$ and antipode $S$.
We will use the Sweedler notation such as
\begin{equation*}
  \Delta(h) = h_{(1)} \otimes h_{(2)}
  \text{\ and \ }
  \Delta(h_{(1)}) \otimes h_{(2)} = h_{(1)} \otimes h_{(2)} \otimes h_{(3)}
  = h_{(1)} \otimes \Delta(h_{(2)})
\end{equation*}
for $h \in H$. To conclude this section, we explain what our results mean if $\mathcal{C}$ is the representation category of $H$.

Recall that a (left-left-){\em Yetter-Drinfeld $H$-module} is a left $H$-module $V$ endowed with a left $H$-comodule structure, denoted by $v \mapsto v_{(-1)} \otimes v_{(0)}$, such that
\begin{equation*}
  (h \cdot v)_{(-1)} \otimes (h \cdot v)_{(0)} = h_{(1)} v_{(-1)} S(h_{(3)}) \otimes h_{(2)} v_{(0)}
\end{equation*}
holds for all $h \in H$ and $v \in V$. A Yetter-Drinfeld module $V$ over $H$ becomes an object of the center of the category $H\mbox{-{\sf Mod}}$ of left $H$-modules by
\begin{equation}
  \label{eq:YD-half-br}
  \sigma_{V,X}: V \otimes X \to X \otimes V,
  \quad v \otimes x \mapsto v_{(-1)} x \otimes v_{(0)}
\end{equation}
for $X \in H\mbox{-{\sf Mod}}$. This construction establishes an isomorphism between the category ${}^H_H \mathcal{YD}$ of Yetter-Drinfeld $H$-modules and $\mathcal{Z}(H\mbox{-{\sf Mod}})$.

If $\mathcal{C} = H\mbox{-{\sf mod}}$ is the category of finite-dimensional left $H$-modules, then $\mathcal{Z}(\mathcal{C})$ can be identified with the category ${}^H_H \mathcal{YD}_{\text{\sf fd}}$ of finite-dimensional Yetter-Drinfeld $H$-modules. Under the identification, the forgetful functor $U: \mathcal{Z}(\mathcal{C}) \to \mathcal{C}$ corresponds to the functor forgetting the left $H$-comodule structure. A right adjoint $R$ of $U$ is given as follows: As a vector space, $R(V) = H \otimes_k V$ for $V \in \mathcal{C}$. The action and the coaction of $H$ on $R(V)$ are given by
\begin{equation*}
  h \cdot (a \otimes v) = h_{(1)} a S(h_{(3)}) \otimes h_{(2)} v
  \quad \text{and} \quad
  a \otimes v \mapsto a_{(1)} \otimes a_{(2)} \otimes v
\end{equation*}
for $a, h \in H$ and $v \in V$. Thus $A = U R(k)$, where $k$ is the trivial $H$-module, can be identified with $H$ as a vector space. The action of $H$ on $A = H$ is the adjoint action $\triangleright$ given by $h \triangleright a = h_{(1)} a S(h_{(2)})$ for $h \in H$ and $a \in A$, and the algebra structure of $A$ is the same as $H$ (see, {\it e.g.}, \cite{2013arXiv1309.4539S} and \cite{2014arXiv1402.3482S} for details). By~\eqref{eq:YD-half-br} and~\eqref{eq:can-act-by-half-br}, the canonical action is given by
\begin{equation}
  \label{eq:can-act-Hopf}
  \rho_X(a \otimes x) = a_{(1)} x \otimes \varepsilon(a_{(2)}) = a x
\end{equation}
for $a \in A$ and $x \in X \in H\mbox{-{\sf mod}}$.

Under the identification $\Hom_k(k, A) \cong A$, the vector space $\CE(\mathcal{C}) \subset A$ coincides with the center of $H$. The bijection \eqref{eq:CE-and-End-id} corresponds to the well-known fact that the center of $H$ is canonically isomorphic to the endomorphism algebra of $\id_{\mathcal{C}}$. By the definition of the duality functor on $\mathcal{C}$, the antipodal map is given by
\begin{equation*}
  \mathfrak{S}(a) = S^{-1}(a) \quad (a \in \CE(\mathcal{C})).
\end{equation*}

An element of $\CF(\mathcal{C})$ is an $H$-linear map $A \to k$. Thus,
\begin{align*}
  \CF(\mathcal{C}) & = \{ f \in H^* \mid \text{$f(h_{(1)} a S(h_{(2)})) = \varepsilon(h) f(a)$ for all $a, h \in H$} \} \\
  & = \{ f \in H^* \mid \text{$f(b a) = f(a S^2(b))$ for all $a, b \in H$} \}.
\end{align*}
The product of $\CF(\mathcal{C})$ is given by the convolution product. To see this, we note that the comultiplication of the comonad $\bar{Z} = U R$ is given by
\begin{equation*}
  \delta_V: \bar{Z}(V) \to \bar{Z}{}^2(V),
  \quad a \otimes v \mapsto a_{(1)} \otimes a_{(2)} \otimes v
  \quad (a \in H, v \in V).
\end{equation*}
Thus, for all $f, g \in \CF(\mathcal{C})$ and $a \in A$, we have
\begin{equation*}
  \langle f \star g, a \rangle
  = \langle f \circ \bar{Z}(g), a_{(1)} \otimes a_{(2)} \rangle
  = \langle f, a_{(1)} \rangle \langle g, a_{(2)} \rangle.
\end{equation*}

Now we suppose that $H$ has a {\em pivotal element} $g \in H$, {\it i.e.}, an invertible element such that $\Delta(g) = g \otimes g$ and $S^2(h) = g h g^{-1}$ for all $h \in H$. Then $\mathcal{C}$ is a pivotal monoidal category with pivotal structure
\begin{equation*}
  j_X: X \to X^{**} \quad (X \in \mathcal{C}),
  \quad \langle j_X(x), p \rangle = \langle p, g x \rangle
  \quad (x \in X, p \in X^*).
\end{equation*}
By~\eqref{eq:can-act-Hopf}, the internal character of $X \in \mathcal{C}$ is given by
\begin{equation*}
  \langle \mathsf{ch}(X), a \rangle
  = \Trace_X(g a)
  \quad (a \in A),
\end{equation*}
where $\Trace_X(h)$ is the trace of the action of $h \in H$ on $X$.

\begin{example}[the Taft algebra]
  \label{ex:Taft}
  Let $N > 1$ be an integer, and let $\omega$ be a primitive $N$-th root of unity.
  The {\em Taft algebra} $T(\omega)$ is the algebra generated by $g$ and $x$ subject to the relations $g^N = 1$, $x^N = 0$ and $g x = \omega x g$.
  The algebra $T(\omega)$ has a Hopf algebra structure determined by
  \begin{gather*}
    \Delta(g) = g \otimes g, \quad
    \Delta(x) = x \otimes g + 1 \otimes x, \\
    \varepsilon(g) = 1, \quad
    \varepsilon(x) = 0, \quad 
    S(g) = g^{-1}, \quad
    S(x) = - x g^{-1}.
  \end{gather*}
  The set $\{ g^i x^j \mid i, j = 0, \dotsc, N - 1 \}$ is a basis of $T(\omega)$.
  If $f: T(\omega) \to k$ is a class function, then we have
  \begin{equation*}
    \langle f, g^i x^j \rangle
    = \langle f, g g^i x^j g^{-1}\rangle
    = \omega^{-j} \langle f, g^i x^j \rangle
  \end{equation*}
  for $i, j = 0, \dotsc, N - 1$. Thus $\langle f, g^i x^j \rangle = 0$ whenever $j > 0$.
  Conversely, a linear functional $f$ on $T(\omega)$ satisfying this property is a class function.
  Thus, in conclusion, the set of class functions on $T(\omega)$ is given by
  \begin{equation*}
    \CF(T(\omega)\mbox{-{\sf mod}})
    = \mathrm{span}_k \, \{ \alpha_i \mid i = 0, \dotsc, N - 1 \},
  \end{equation*}
  where $\alpha_i: T(\omega) \to k$ is the algebra map such that $\langle \alpha_i, g \rangle = \omega^i$ and $\langle \alpha_i, x \rangle = 0$.
\end{example}

\section{Linear independence of irreducible characters}
\label{sec:indep}

\subsection{Finite tensor categories}
\label{subsec:FTC}

Throughout this section, we work over an algebraically closed field $k$.
A {\em finite abelian category} (over $k$) is a $k$-linear category that is equivalent to $A\mbox{-{\sf mod}}$ for some finite-dimensional $k$-algebra $A$.
A {\em finite tensor category} \cite{MR2119143} is a rigid monoidal category $\mathcal{C}$ such that
\begin{enumerate}
\item $\mathcal{C}$ is a finite abelian category,
\item the tensor product $\otimes: \mathcal{C} \times \mathcal{C} \to \mathcal{C}$ is $k$-linear in each variable, and
\item $\End_{\mathcal{C}}(\unitobj) \cong k$ as algebras.
\end{enumerate}
In view of (3), we often identify $\End_{\mathcal{C}}(\unitobj)$ with $k$.
Thus, in particular, the pivotal trace and the pivotal dimension in a pivotal finite tensor category are regarded as elements of $k$.

For a finite abelian category $\mathcal{A}$, we denote by $\Gr(\mathcal{A})$ its Grothendieck group and set
$\Gr_k(\mathcal{A}) = k \otimes_{\mathbb{Z}} \Gr(\mathcal{A})$. It is well-known that the isomorphism classes of simple objects of $\mathcal{A}$ is a basis of $\Gr_k(\mathcal{A})$.
If $\mathcal{C}$ is a finite tensor category, then $\Gr_k(\mathcal{C})$ is an algebra with respect to the multiplication $[V] \cdot [W] = [V \otimes W]$.
Thus, if this is the case, we call $\Gr_k(\mathcal{C})$ the {\em Grothendieck algebra}.

Following \cite[\S5]{MR1862634}, the coend \eqref{eq:Hopf-monad-Z} exists for all $V \in \mathcal{C}$ if $\mathcal{C}$ is a finite tensor category (see also the discussion in \S\ref{subsec:integ-coref}).
In the rest of this paper, we study the internal characters in a finite tensor category and, especially, its relation to the Grothendieck algebra.

\begin{notation*}
  For convenience, we gather notations used in \S\ref{sec:internal-ch}.
  First, we denote by $U: \mathcal{Z}(\mathcal{C}) \to \mathcal{C}$ and $R$ the forgetful functor and its right adjoint, respectively.
  The functor $\bar{Z} = U R$ is a Hopf comonad on $\mathcal{C}$.
  Since $\bar{Z}$ can be given by \eqref{eq:Hopf-comonad-Z},
  there is a universal dinatural transformation
  \begin{equation*}
    \pi_{V; X}: \bar{Z}(V) \to X \otimes V \otimes X^* \quad (V, X \in \mathcal{C}).
  \end{equation*}
  The comultiplication $\delta$, the counit $\varepsilon$ and the monoidal structure
  \begin{equation*}
    \bar{Z}_0: \unitobj \to \bar{Z}(\unitobj)
    \quad \text{and}
    \quad \bar{Z}_2(X, Y): \bar{Z}(X) \otimes \bar{Z}(Y) \to \bar{Z}(X \otimes Y)
    \quad (X, Y \in \mathcal{C})
  \end{equation*}
  of $\bar{Z}$ are expressed in terms of the universal dinatural transformation in a similar way as the central Hopf monad; see \eqref{eq:Hopf-comonad-Z-comult}--\eqref{eq:adj-obj-mult}.

  The object $A := \bar{Z}(\unitobj)$ is an algebra in $\mathcal{C}$, called the {\em adjoint algebra},
  with multiplication $m = \bar{Z}_2(\unitobj, \unitobj)$
  and unit $u = \bar{Z}_0$.
  The morphism $\pi_{\unitobj; X}$ induces the {\em canonical action} $\rho_X: A \otimes X \to X$ for $X \in \mathcal{C}$,
  and the {\em internal character} $\mathsf{ch}(X): A \to \unitobj$ of $X$ is defined to be the partial pivotal trace of $\rho_X$ (if $\mathcal{C}$ is pivotal).

  Finally, we set $\CE(\mathcal{C}) = \Hom_{\mathcal{C}}(\unitobj, A)$ and $\CF(\mathcal{C}) = \Hom_{\mathcal{C}}(A, \unitobj)$.
  It is obvious that the multiplications of $\CE(\mathcal{C})$ and $\CF(\mathcal{C})$ are $k$-linear.
  Thus these monoids are in fact finite-dimensional algebras.
  The algebra $\CE(\mathcal{C})$ is always commutative, while $\CF(\mathcal{C})$ may not.
\end{notation*}

\subsection{Linear independence of irreducible characters}

Let $\mathcal{C}$ be a finite tensor category over $k$ endowed with a pivotal structure $j: \id_{\mathcal{C}} \to (-)^{**}$, and let $\{ V_0, \dotsc, V_m \}$ be a complete set of representatives of isomorphism classes of simple objects of $\mathcal{C}$.
An element of the set
\begin{equation*}
  \mathrm{Irr}(\mathcal{C}) := \{ \mathsf{ch}(V_0), \dotsc, \mathsf{ch}(V_m) \} \subset \CF(\mathcal{C})
\end{equation*}
is called an {\em irreducible character}. The main result of this section is:

\begin{theorem}
  \label{thm:irr-ch-indep}
  The set $\mathrm{Irr}(\mathcal{C})$ is linearly independent in $\CF(\mathcal{C})$.
\end{theorem}

We give some immediate consequences of this theorem.
Recall that the internal character $\mathsf{ch}(X)$ of $X \in \mathcal{C}$ is defined to be the partial pivotal trace of $\rho_X$.
Since $\rho_X$ is natural in $X$, and since the partial pivotal trace is additive with respect to exact sequences, the linear map
\begin{equation*}
  \mathsf{ch}: \Gr_k(\mathcal{C}) \to \CF(\mathcal{C}),
  \quad [X] \mapsto \mathsf{ch}(X)
  \quad (X \in \mathcal{C})
\end{equation*}
is well-defined. Moreover, Theorem~\ref{thm:ch-mult} implies that this map is a homomorphism of algebras.
Since the set $\{ [V_0], \dotsc, [V_m] \}$ is a basis of $\Gr_k(\mathcal{C})$, we have:

\begin{corollary}
  \label{cor:irr-ch-indep-cor-1}
  The map $\mathsf{ch}: \Gr_k(\mathcal{C}) \to \CF(\mathcal{C})$ is an injective $k$-algebra map.
\end{corollary}

By Theorem~\ref{thm:class-ft-induc}, we also have the following corollary:

\begin{corollary}
  \label{cor:irr-ch-indep-cor-2}
  The algebra $\Gr_k(\mathcal{C})$ can be embedded into $\End_{\mathcal{Z}(\mathcal{C})}(R(\unitobj))$.
\end{corollary}

Before giving a proof of Theorem~\ref{thm:irr-ch-indep}, we observe two cases that the linear independence of the irreducible characters is relatively obvious.

\begin{example}
  \label{ex:indep-irr-ch-fusion}
  A {\em fusion category} \cite{MR2183279} is a semisimple finite tensor category.
  If $\mathcal{C}$ in the above is a fusion category, then, by the argument of \cite[\S5.1.3]{MR1862634},
  \begin{equation*}
    \bar{Z}(V) = \bigoplus_{i = 0}^m V_i \otimes V \otimes V_i^*
  \end{equation*}
  as an object in $\mathcal{C}$ and the morphism $\pi_{V; X}$ is just the projection if $X$ is one of $V_0, \dotsc, V_m$. Hence, by~\eqref{eq:character-pi}, $\mathsf{ch}(V_i) \in \CF(\mathcal{C})$ corresponds to $\widetilde{\eval}_{V_i}$ via
  \begin{equation*}
    \CF(\mathcal{C}) = \Hom_{\mathcal{C}}(A, \unitobj) \cong \bigoplus_{i = 0}^m \Hom_{\mathcal{C}}(V_i \otimes V_i^*, \unitobj).
  \end{equation*}
  By Schur's lemma, $\Hom_{\mathcal{C}}(V_i \otimes V_i^*, \unitobj)$ is one-dimensional. Thus the set $\mathrm{Irr}(\mathcal{C})$ is not only linearly independent, but also a basis of $\CF(\mathcal{C})$.
\end{example}

\begin{remark}
  \label{rem:Ost-and-NY}
  Thus, for a pivotal fusion category $\mathcal{C}$, we have an isomorphism
  \begin{equation*}
    \Gr_k(\mathcal{C}) \cong \End_{\mathcal{Z}(\mathcal{C})}(R(\unitobj))
  \end{equation*}
  of algebras by Theorem~\ref{thm:class-ft-induc}.
  We shall note some related results:
  First, in the case where $k$ is of characteristic zero, the existence of such an isomorphism has been showed by Ostrik \cite[Corollary 2.16]{2013arXiv1309.4822O}.
  His method cannot be applied in the case of positive characteristic.
  Our construction of the isomorphism is somewhat more canonical and also works in such a case.
  We also note that Neshveyev and Yamashita \cite{2015arXiv150107390N} have established such an isomorphism for $C^*$-tensor categories possibly with infinitely many isomorphism classes of simple objects.
  Our approach is basically same as \cite{2015arXiv150107390N} if it is limited to $C^*$-fusion categories.
\end{remark}

\begin{example}
  \label{ex:indep-irr-ch-Hopf}
  Let $H$ be a finite-dimensional Hopf algebra over $k$ with pivotal element $g$,
  and let $V_0, \dotsc, V_m$ be the complete set of representatives of isomorphism classes of simple left $H$-modules.
  As explained in \S\ref{subsec:Hopf-alg}, then $\mathcal{C} = H\mbox{-{\sf mod}}$ is a pivotal finite tensor category.
  Since $g$ is invertible, Theorem~\ref{thm:irr-ch-indep} is equivalent to that
  \begin{equation}
    \label{eq:indep-irr-ch-Hopf}
    \text{$\{ \Trace_{V_i} \}_{i = 0}^m$ is linearly independent in $H^*$}
  \end{equation}
  in this particular case. This statement is well-known to be true and can be proved as follows: Set $H^{\mathrm{ss}} = H/J$, where $J$ is the Jacobson radical of $H$. By the Artin-Wedderburn theorem, there is an isomorphism
  \begin{equation}
    \label{eq:ss-quot-decomp}
    H^{\mathrm{ss}} \cong \End_k(V_0) \oplus \dotsb \oplus \End_k(V_m)
  \end{equation}
  of algebras. Since $\Trace_{V_i}$ coincides with the composition
  \begin{equation}
    \label{eq:ss-quot-ch}
    H \xrightarrow{\ \text{quotient} \ }
    H^{\mathrm{ss}}
    \cong \bigoplus_{i = 0}^m \End_k(V_i)
    \xrightarrow{\ \text{projection} \ }
    \End_k(V_i)
    \xrightarrow{\ \text{trace} \ } k,
  \end{equation}
  the assertion~\eqref{eq:indep-irr-ch-Hopf} follows.
\end{example}

Our proof of Theorem~\ref{thm:irr-ch-indep} basically follows the above scheme.
More precisely, we will introduce an object $A^{\mathrm{ss}} \in \mathcal{C}$,
which can be considered as a category-theoretical analogue of $H^{\mathrm{ss}}$ in Example~\ref{ex:indep-irr-ch-Hopf}.
We then show that $A^{\mathrm{ss}}$ is a quotient of $A$ of the form like~\eqref{eq:ss-quot-decomp}
and every irreducible character factors through the quotient morphism $A \twoheadrightarrow A^{\mathrm{ss}}$ like~\eqref{eq:ss-quot-ch}.
Unlike Examples~\ref{ex:indep-irr-ch-fusion} and~\ref{ex:indep-irr-ch-Hopf}, there are some technical difficulties in the general setting.
To prove Theorem~\ref{thm:irr-ch-indep}, we need to go back to the proof of the existence of certain (co)ends including \eqref{eq:Hopf-monad-Z}.

\subsection{Existence of certain (co)ends}
\label{subsec:integ-coref}

Let $\mathcal{C}$ be a finite tensor category over $k$.
For $X \in \mathcal{C}$ and $K \in k\mbox{-{\sf mod}}$, their tensor product $K \cdot X \in \mathcal{C}$ (also called the copower) is defined by
\begin{equation}
  \label{eq:copower}
  \Hom_{\mathcal{C}}(K \cdot X, Y) \cong \Hom_k(K, \Hom_{\mathcal{C}}(X, Y))
\end{equation}
for $Y \in \mathcal{C}$. Let $\LEX(\mathcal{C})$ be the category of $k$-linear left exact endofunctors on $\mathcal{C}$. As explained in \cite{2014arXiv1402.3482S}, the Eilenberg-Watts theorem implies that the functor
\begin{equation}
  \label{eq:equiv-Phi}
  \Phi: \mathcal{C} \boxtimes \mathcal{C} \to \LEX(\mathcal{C}),
  \quad X \boxtimes Y \mapsto \Hom_{\mathcal{C}}(X^*, -) \cdot Y
  \quad (X, Y \in \mathcal{C})
\end{equation}
is an equivalence of $k$-linear categories, where $\boxtimes$ means the Deligne tensor product of $k$-linear abelian categories \cite[\S5]{MR1106898}. As shown in \cite{2014arXiv1402.3482S}, a quasi-inverse of $\Phi$, which we denote by $\overline{\Phi}$, is given by
\begin{equation}
  \label{eq:coend-Phi-inv}
  \overline{\Phi}(F) = \int^{X \in \mathcal{C}} {}^* \! X \boxtimes F(X)
\end{equation}
for $F \in \LEX(\mathcal{C})$. Thus, there are natural isomorphisms
\begin{equation}
  \label{eq:coend-Yoneda}
  F \cong \Phi \overline{\Phi}(F)
  \cong \int^{X \in \mathcal{C}} \Phi({}^* \! X \boxtimes F(X))
  = \int^{X \in \mathcal{C}} \Hom_{\mathcal{C}}(X, -) \cdot F(X)
\end{equation}
for $F \in \LEX(\mathcal{C})$.

Now let $\mathcal{D}$ be a full subcategory of $\mathcal{C}$ closed under taking subobjects, quotient objects and finite direct sums.
Then $\mathcal{D}$ is an abelian category such that the inclusion functor $i_{\mathcal{D}}: \mathcal{D} \hookrightarrow \mathcal{C}$ is exact.
For every $X \in \mathcal{C}$, there is the largest subobject $t_{\mathcal{D}}(X)$ of $X$ belonging to $\mathcal{D}$.
The assignment $X \mapsto t_{\mathcal{D}}(X)$ extends to a $k$-linear left exact functor from $\mathcal{C}$ to $\mathcal{D}$.
Moreover, we have
\begin{equation*}
  \Hom_{\mathcal{C}}(i_{\mathcal{D}}(V), X)
  = \Hom_{\mathcal{D}}(V, t_{\mathcal{D}}(X))
\end{equation*}
for all $V \in \mathcal{D}$ and $X \in \mathcal{C}$.
In what follows, we often omit writing $i_{\mathcal{D}}$ and regard $t_{\mathcal{D}}$ as an endofunctor on $\mathcal{C}$ for simplicity of notation.
By applying Lemma~\ref{lem:coend-adj-2} to the adjunction $i_{\mathcal{D}} \dashv t_{\mathcal{D}}$, we have
\begin{equation}
  \label{eq:coend-SS}
  t_{\mathcal{D}}
  \xrightarrow[\cong]{\ \text{\eqref{eq:coend-Yoneda}} \ }
  \int^{X \in \mathcal{C}} \Hom_{\mathcal{C}}(X, -) \cdot t_{\mathcal{D}}(X)
  \xrightarrow[\cong]{\ \eqref{eq:coend-iso-1} \ }
  \int^{X \in \mathcal{D}} \Hom_{\mathcal{C}}(X, -) \cdot X.
\end{equation}

\begin{lemma}
  \label{lem:integ-coref-subcat-1}
  There is a commutative diagram
  \begin{equation*}
    \minCDarrowwidth100pt
    \begin{CD}
      t_{\mathcal{D}}
      @>{\text{\eqref{eq:coend-SS}}}>>
      \int^{X \in \mathcal{D}} \Hom_{\mathcal{C}}(X, -) \cdot X \\
      @V{\iota}VV @VV{\phi}V \\
      \id_{\mathcal{C}}
      @>>{\text{\rm \eqref{eq:coend-Yoneda} with $F = \id_{\mathcal{C}}$}}>
      \int^{X \in \mathcal{C}} \Hom_{\mathcal{C}}(X, -) \cdot X
    \end{CD}
  \end{equation*}
  in $\LEX(\mathcal{C})$, where $\iota: t_{\mathcal{D}} \to \id_{\mathcal{C}}$ is the counit of $i_{\mathcal{D}} \dashv t_{\mathcal{D}}$ and $\phi$ is the canonical morphism obtained by the universal property of the coend.
\end{lemma}
\begin{proof}
  We consider the following diagram:
  \begin{equation*}
    \begin{CD}
      t_{\mathcal{D}}
      @>{\text{\eqref{eq:coend-Yoneda}}}>>
      \int^{X \in \mathcal{C}} \Hom_{\mathcal{C}}(X, -) \cdot t_{\mathcal{D}}(X)
      @>{\eqref{eq:coend-iso-1}}>>
      \int^{X \in \mathcal{D}} \Hom_{\mathcal{C}}(X, -) \cdot X\phantom{.} \\
      @V{\iota}VV @VV{\int^{X \in \mathcal{C}} \Hom_{\mathcal{C}}(\id_X, -) \cdot \iota_X}V @VV{\phi}V \\
      \id_{\mathcal{C}}
      @>{\text{\eqref{eq:coend-Yoneda}}}>>
      \int^{X \in \mathcal{C}} \Hom_{\mathcal{C}}(X, -) \cdot X
      @=
      \int^{X \in \mathcal{C}} \Hom_{\mathcal{C}}(X, -) \cdot X.
    \end{CD}
  \end{equation*}
  This diagram commutes by the naturality of \eqref{eq:coend-Yoneda} and Lemma~\ref{lem:coend-adj-2}.
  The composition along the top row is \eqref{eq:coend-SS}. Thus the proof is done.
\end{proof}

There is a right action $\triangleleft$ of $\mathcal{C} \boxtimes \mathcal{C}^{\rev}$ on $\mathcal{C}$ given by
$V \triangleleft (X \boxtimes Y) = Y \otimes V \otimes X$
for $V, X, Y \in \mathcal{C}$.
Using this action, we define the functor $\Psi: \LEX(\mathcal{C}) \times \mathcal{C} \to \mathcal{C}$ by
\begin{equation*}
  \Psi(F, V) = V \triangleleft \overline{\Phi}(F)
  \quad (F \in \LEX(\mathcal{C}), V \in \mathcal{C}).
\end{equation*}
Since the base field $k$ is algebraically closed, the action $\triangleleft$ is exact in each variable \cite[Proposition 5.13]{MR1106898}.
Thus we set
$Z(\mathcal{D}; V) = \Psi(t_{\mathcal{D}}, V)$
for $\mathcal{D}$ as above and $V \in \mathcal{C}$. Then, by~\eqref{eq:coend-Yoneda} and~\eqref{eq:coend-SS}, we have, symbolically,
\begin{equation}
  \label{eq:coend-Z-D-V}
  Z(\mathcal{D}; V) = \int^{X \in \mathcal{D}} X \otimes V \otimes {}^* \! X.
\end{equation}

\begin{lemma}
  \label{lem:integ-coref-subcat-2}
  The morphism $\phi_V: Z(\mathcal{D}; V) \to Z(\mathcal{C}; V)$ in $\mathcal{C}$ obtained by the universal property of the coend \eqref{eq:coend-Z-D-V} is monic.
\end{lemma}
\begin{proof}
  By the above construction of $Z(\mathcal{D}; V)$, we have $\phi_V = \Psi(\iota, \id_V)$, where $\iota$ is the counit of $i_{\mathcal{D}} \dashv t_{\mathcal{D}}$.
  We note that $\iota$ is the inclusion $t_{\mathcal{D}} \hookrightarrow \id_{\mathcal{C}}$ and, in particular, is monic in $\LEX(\mathcal{C})$.
  Since $\Psi$ is exact in each variable, $\phi_V$ is monic.
\end{proof}

We package the above results toward the proof of Theorem~\ref{thm:irr-ch-indep}.

\begin{lemma}
  \label{lem:integ-coref-subcat-3}
  Let $\mathcal{D}$ be a full subcategory of $\mathcal{C}$ closed under subobjects, quotient objects and direct sums.
  Then the morphism
  \begin{equation*}
    q: \int_{X \in \mathcal{C}} X \otimes X^* \to \int_{X \in \mathcal{D}} X \otimes X^*
  \end{equation*}
  obtained by the universal property of the right end is an epimorphism.
\end{lemma}
\begin{proof}
  The universal dinatural transformation $j_{V; X}: X \otimes V \otimes {}^* \! X \to Z(\mathcal{D}; V)$ of the coend \eqref{eq:coend-Z-D-V} induces a family
  \begin{equation*}
    Z(\mathcal{D}; \unitobj)^* \xrightarrow{\quad (j_{\unitobj; X})^* \quad}
    (X \otimes \unitobj \otimes {}^* \! X)^* = X \otimes X^*
  \end{equation*}
  of morphisms in $\mathcal{C}$ that is dinatural in $X \in \mathcal{D}$.
  Since the left duality functor is an anti-equivalence, we have, symbolically,
  \begin{equation*}
    \int_{X \in \mathcal{D}} X \otimes X^* = Z(\mathcal{D}; \unitobj)^*
    \quad \text{and} \quad
    \int_{X \in \mathcal{C}} X \otimes X^* = Z(\mathcal{C}; \unitobj)^*.
  \end{equation*}
  Thus $q = \phi_{\unitobj}^*$ with the notation of Lemma~\ref{lem:integ-coref-subcat-2}.
  Since $\phi_{\unitobj}$ is monic, $q$ is epic.
\end{proof}

\subsection{Proof of Theorem~\ref{thm:irr-ch-indep}}

We now give a proof of Theorem~\ref{thm:irr-ch-indep}. Let $\mathcal{C}$ be a pivotal finite tensor category over $k$, and let $\mathcal{S}$ be the full subcategory of $\mathcal{C}$ consisting of all semisimple objects.
As a categorical analogue of $H^{\mathrm{ss}}$ in Example~\ref{ex:indep-irr-ch-Hopf}, we propose to consider the following object:

\begin{definition}
  $\displaystyle A^{\mathrm{ss}} := \int_{X \in \mathcal{S}} X \otimes X^*$.
\end{definition}

As before, we let $\{ V_0, \dotsc, V_m \}$ be the complete set of representatives of isomorphism classes of simple objects of $\mathcal{C}$.
By \cite[\S5.1.3]{MR1862634}, we may assume that
\begin{equation*}
  A^{\mathrm{ss}} = \bigoplus_{i = 0}^m V_i \otimes V_i^*
\end{equation*}
and the universal dinatural transformation $\pi_X^{\mathrm{ss}}: A^{\mathrm{ss}} \to X \otimes X^*$ ($X \in \mathcal{S}$) is just the projection if $X$ is one of $V_0, \dotsc, V_m$.

\begin{proof}[Proof of Theorem~\ref{thm:irr-ch-indep}]
  By the universal property, there exists a unique morphism $q: A \to A^{\mathrm{ss}}$ such that $\pi_{\unitobj; X} = \pi_X^{\mathrm{ss}} \circ q$ for all $X \in \mathcal{S}$. Thus, by~\eqref{eq:character-pi},
  \begin{equation*}
    \mathsf{ch}(V_i) = \widetilde{\eval}_{V_i} \circ \pi_{\unitobj; V_i} = \widetilde{\eval}_{V_i} \circ \pi_{V_i}^{\mathrm{ss}} \circ q
  \end{equation*}
  for $i = 0, \dotsc, m$. This means that $\mathsf{ch}(V_i)$ is obtained as the image of $\widetilde{\eval}_{V_i}$ under
  \begin{equation*}
    \bigoplus_{i = 0}^m \Hom_{\mathcal{C}}(V_i \otimes V_i^*, \unitobj)
    \cong \Hom_{\mathcal{C}}(A^{\mathrm{ss}}, \unitobj)
    \xrightarrow{\quad \Hom_{\mathcal{C}}(q, \unitobj) \quad}
    \Hom_{\mathcal{C}}(A, \unitobj) = \CF(\mathcal{C}).
  \end{equation*}
  It is easy to see that $\mathcal{D} = \mathcal{S}$ satisfies the assumption of Lemma~\ref{lem:integ-coref-subcat-3}.
  Thus $q$ is epic by that lemma, and hence $\Hom_{\mathcal{C}}(q, \unitobj)$ is injective.
  This implies that the set $\mathrm{Irr}(\mathcal{C})$ is linearly independent.
\end{proof}

\subsection{Further remarks on coends}

The equivalence $\Phi$, introduced in \S\ref{subsec:integ-coref}, seems to be an essential tool to study properties of certain coends.
Before closing this section, we add another application of the equivalence $\Phi$.
Let $\mathcal{C}$ and $\{ V_0, \dotsc, V_m \}$ be as above (but now $\mathcal{C}$ is not necessarily pivotal).
With motivation coming from logarithmic conformal field theory, the character of a certain coend has been studied in \cite{MR3289342,MR3146014}.
In relation to this, we give the following formula of the composition factors of such a coend:

\begin{theorem}
  \label{thm:comp-factor-coends}
  For all $k$-linear left exact functor $F: \mathcal{C} \to \mathcal{C}$, we have
  \begin{equation}
    \label{eq:compo-factor-coend}
    \left[ \int^{X \in \mathcal{C}} {}^* \! X \boxtimes F(X) \right]
    = \sum_{i, j = 0}^m [F(P_i^*):V_j] \cdot [V_i \boxtimes V_j]
  \end{equation}
  in $\Gr(\mathcal{C} \boxtimes \mathcal{C})$, where $P_i$ is the projective cover of $V_i$ and $[X : V_i]$ is the multiplicity of $V_i$ in the composition series of $X$.
\end{theorem}
\begin{proof}
  Let $\mathcal{A}$ be a finite abelian category over $k$. We first recall that if $P$ is the projective cover of a simple object $V \in \mathcal{A}$, then there holds
  \begin{equation}
    \label{eq:compo-factor-pf-1}
    [X:V] = \dim_k \Hom_{\mathcal{A}}(P, X)
  \end{equation}
  for all $X \in \mathcal{A}$.

  Let $\Phi: \mathcal{C} \boxtimes \mathcal{C} \to \LEX(\mathcal{C})$ be the equivalence given by~\eqref{eq:equiv-Phi}, and let $\overline{\Phi}$ be the quasi-inverse of $\Phi$ given by~\eqref{eq:coend-Phi-inv}. The left-hand side of~\eqref{eq:compo-factor-coend} is $[\overline{\Phi}(F)]$. Since every simple object of $\mathcal{C} \boxtimes \mathcal{C}$ is of the form $V_i \boxtimes V_j$, and since $P_i \boxtimes P_j$ is the projective cover of $V_i \boxtimes V_j$, we have
  \begin{equation}
    \label{eq:compo-factor-pf-2}
    [\overline{\Phi}(F)] = \sum_{i, j = 0}^m \dim_k(\Hom_{\mathcal{C} \boxtimes \mathcal{C}}(P_i \boxtimes P_j, \overline{\Phi}(F))) \cdot [V_i \boxtimes V_j]
  \end{equation}
  for all $F \in \LEX(\mathcal{C})$ in $\Gr(\mathcal{C} \boxtimes \mathcal{C})$ by~\eqref{eq:compo-factor-pf-1}. For all $i, j = 0, \dotsc, m$, we compute
  \begin{align*}
    \Hom_{\mathcal{C}}(P_i \boxtimes P_j, \overline{\Phi}(F))
    & \cong \Nat(\Phi(P_i \boxtimes P_j), F) \\
    & \cong \textstyle \int_{X \in \mathcal{C}}
    \Hom_{\mathcal{C}}(\Hom_{\mathcal{C}}(P_i^*, X) \cdot P_j, F(X)) \\
    & \cong \textstyle \int_{X \in \mathcal{C}}
    \Hom_{k}(\Hom_{\mathcal{C}}(P_i^*, X), \Hom_{\mathcal{C}}(P_j, F(X))) \\
    & \cong \Nat(\Hom_{\mathcal{C}}(P_i^*, -), \Hom_{\mathcal{C}}(P_j, F(-))) \\
    & \cong \Hom_{\mathcal{C}}(P_j, F(P_i^*))
    \quad (\text{by the Yoneda lemma}).
  \end{align*}
  The claim now follows from \eqref{eq:compo-factor-pf-1} and \eqref{eq:compo-factor-pf-2}.
\end{proof}

\section{Integrals and Fourier transform}
\label{sec:integral}

\subsection{Unimodular finite tensor categories}
\label{subsec:unimoftc}

Throughout this section, we work over an algebraically closed field $k$. Etingof, Nikshych and Ostrik \cite{MR2097289} introduced the {\em distinguished invertible object} $D \in \mathcal{C}$ of a finite tensor category $\mathcal{C}$ over $k$. The object $D$ is a category-theoretical analogue of the modular function (also called the distinguished grouplike element) of a finite-dimensional Hopf algebra, and therefore we say that $\mathcal{C}$ is {\em unimodular} if $D \cong \unitobj$.

The aim of this section is to introduce an {\em integral}, a {\em cointegral} and the {\em Fourier transform} for a unimodular finite tensor category.
We do not recall the definition of the distinguished invertible object. Instead, we will use the following characterization of the unimodularity given in \cite{2014arXiv1402.3482S}.

\begin{theorem}
  \label{thm:unimo}
  Let $\mathcal{C}$ be a finite tensor category over $k$, and use the set of notations in \S\ref{subsec:FTC}. Then the following assertions are equivalent:
  \begin{enumerate}
  \item $\mathcal{C}$ is unimodular.
  \item $U$ is a Frobenius functor, i.e., $R$ is also left adjoint to $U$.
  \item $\Hom_{\mathcal{Z}(\mathcal{C})}(R(\unitobj), \unitobj) \ne 0$.
  \end{enumerate}
\end{theorem}

We also note that $\mathbf{A} := R(\unitobj)$ is an algebra in $\mathcal{Z}(\mathcal{C})$ as the image of $\unitobj \in \mathcal{C}$ under a monoidal functor.
Since the adjoint algebra is given by $A = U(\mathbf{A})$, the structure morphisms of $\mathbf{A}$ are expressed by the same symbol as those of $A$.
It is known that the algebra $\mathbf{A}$ is commutative in the sense that
\begin{equation}
  \label{eq:adj-alg-comm}
  m \circ \sigma_{\mathbf{A}, \mathbf{A}} = m,
\end{equation}
where $\sigma$ is the braiding of $\mathcal{Z}(\mathcal{C})$; see, {\it e.g.}, \cite{2014arXiv1402.3482S}.

Now suppose that $\mathcal{C}$ is unimodular. Then we have
\begin{equation}
  \label{eq:Fb-tr-uniq}
  \Hom_{\mathcal{Z}(\mathcal{C})}(\mathbf{A}, \unitobj_{\mathcal{Z}(\mathcal{C})})
  \cong \Hom_{\mathcal{C}}(\unitobj, U(\unitobj_{\mathcal{Z}(\mathcal{C})}))
  = \Hom_{\mathcal{C}}(\unitobj, \unitobj)
  \cong k
\end{equation}
by Theorem~\ref{thm:unimo}.
Thus there is a non-zero morphism $\lambda: \mathbf{A} \to \unitobj$ in $\mathcal{Z}(\mathcal{C})$ and such a morphism is unique up to scalar multiple.
It is also shown in \cite{2014arXiv1402.3482S} that $\mathbf{A}$ is a Frobenius algebra in $\mathcal{Z}(\mathcal{C})$ with trace $\lambda$ in the sense that
\begin{equation}
  \label{eq:Fb-iso}
  \mathbf{A} \xrightarrow{\quad \id_\mathbf{A} \otimes \coev_\mathbf{A} \quad}
  \mathbf{A} \otimes \mathbf{A} \otimes \mathbf{A}^*
  \xrightarrow{\quad m \otimes \id_{\mathbf{A}^*} \quad}
  \mathbf{A} \otimes \mathbf{A}^*
  \xrightarrow{\quad \lambda \otimes \id_{\mathbf{A}^*} \quad}
  \mathbf{A}^*
\end{equation}
is an isomorphism in $\mathcal{Z}(\mathcal{C})$.

For later use, we note some consequences of the Frobenius property of $\mathbf{A}$ (see, {\it e.g.}, \cite{MR2500035} for the general theory of Frobenius algebras in a rigid monoidal category).
Fix a non-zero morphism $\lambda: \mathbf{A} \to \unitobj$ in $\mathcal{Z}(\mathcal{C})$ and set
\begin{equation}
  \label{eq:Fb-dual-basis}
  e_{\lambda} = \lambda \circ m
  \quad \text{and} \quad
  d_{\lambda} = (\id_{\mathbf{A}} \otimes \phi_{\lambda}^{-1}) \circ \coev_{\mathbf{A}},
\end{equation}
where $\phi_{\lambda}: \mathbf{A} \to \mathbf{A}^*$ is the isomorphism given by~\eqref{eq:Fb-iso}.
Then the triple $(\mathbf{A}, e_{\lambda}, d_{\lambda})$ is a left dual object of $\mathbf{A}$.
The $\mathbf{A}$-linearity of $\phi_{\lambda}$ implies
\begin{equation}
  \label{eq:Fb-comul}
  (\id_{\mathbf{A}} \otimes m) \circ (d_{\lambda} \otimes \id_{\mathbf{A}})
  = (m \otimes \id_{\mathbf{A}}) \circ (\id_{\mathbf{A}} \otimes d_{\lambda}).
\end{equation}
Let $\Delta: \mathbf{A} \to \mathbf{A} \otimes \mathbf{A}$ be the left-hand of \eqref{eq:Fb-comul}. Then $(\mathbf{A}, \Delta, \lambda)$ is a coalgebra in $\mathcal{Z}(\mathcal{C})$. Namely, the following equations hold:
\begin{gather*}
  (\Delta \otimes \id_{\mathbf{A}}) \circ \Delta
  = (\id_{\mathbf{A}} \otimes \Delta) \circ \Delta, \quad
  (\lambda \otimes \id_{\mathbf{A}}) \circ \Delta
  = \id_{\mathbf{A}} =
  (\id_{\mathbf{A}} \otimes \lambda) \circ \Delta.
\end{gather*}

\subsection{Characterization of trivial objects}

Let $\mathcal{C}$ be a finite tensor category over $k$ (the unimodularity and the pivotality are not needed for a while).
We say that an object $X \in \mathcal{C}$ is {\em trivial} if it is isomorphic to the direct sum of finitely many (possibly zero) copies of the unit object.
Before we introduce the notions of integrals and cointegrals, we provide the following characterization of trivial objects:

\begin{proposition}
  \label{prop:triv-obj}
  An object $X \in \mathcal{C}$ is trivial if and only if
  \begin{equation}
    \label{eq:triv-obj-rho}
    \rho_X = \varepsilon_{\unitobj} \otimes \id_X.
  \end{equation}
\end{proposition}
\begin{proof}
  It is easy to see that an object $X \in \mathcal{C}$ satisfies \eqref{eq:triv-obj-rho} if $X$ is trivial.  To prove the converse, we recall that $F := \int^{X \in \mathcal{C}} X^* \otimes X$ has the coalgebra structure such that $F^* \cong A$ as algebras in $\mathcal{C}$ (see Remark~\ref{rem:coend-Hopf}). By Lyubashenko's result on the category of $F$-comodules \cite[\S2.7]{MR1625495}, we see that the functor
  \begin{equation*}
    \mathcal{C} \boxtimes \mathcal{C} \to \text{(the category of left $A$-modules in $\mathcal{C}$)},
    \quad V \boxtimes W \mapsto (V \otimes W, \rho_V \otimes \id_W)
  \end{equation*}
  is an equivalence of categories. Now suppose that $X \in \mathcal{C}$ satisfies~\eqref{eq:triv-obj-rho}. By the above equivalence, we have $\unitobj \boxtimes X \cong X \boxtimes \unitobj$ in $\mathcal{C} \boxtimes \mathcal{C}$. Applying the functor
  \begin{equation*}
    \mathcal{C} \boxtimes \mathcal{C} \to \mathcal{C},
    \quad V \boxtimes W \mapsto \Hom_{\mathcal{C}}(\unitobj, V) \cdot W
  \end{equation*}
  to both sides, we get $X \cong \Hom_{\mathcal{C}}(\unitobj, X) \cdot \unitobj$. Hence $X$ is trivial. The proof is done.
\end{proof}

\subsection{Integral theory}
\label{subsec:int-theory}

We introduce an integral and a cointegral in a unimodular finite tensor category. From now on, we suppose that $\mathcal{C}$ is unimodular and use the set of notation given in \S\ref{subsec:FTC}.

\begin{definition}
  \label{def:integ}
  An {\em integral} in $\mathcal{C}$ is a morphism $\Lambda: \unitobj \to A$ in $\mathcal{C}$ such that
  \begin{equation}
    \label{eq:integ-def}
    m \circ (\id_A \otimes \Lambda) = \varepsilon_{\unitobj} \otimes \Lambda.
  \end{equation}
  A {\em cointegral} in $\mathcal{C}$ is a morphism $\lambda: A \to \unitobj$ such that
  \begin{equation}
    \label{eq:coint-def}
    \bar{Z}(\lambda) \circ \delta_{\unitobj} = u \circ \lambda.
  \end{equation}
\end{definition}

\begin{remark}
  Definition~\ref{def:integ} makes sense even if $\mathcal{C}$ is not unimodular. The unimodularity will be used to show that non-zero (co)integrals exist. We suspect that a non-zero (co)integral in a finite tensor category $\mathcal{A}$ exists only if $\mathcal{A}$ is unimodular ({\it cf}. the case of Hopf algebras, \S\ref{subsec:Hopf-alg-Fourier}). The above definition should be modified to a form involving the distinguished invertible object for the non-unimodular case.
\end{remark}

If we identify $\mathcal{Z}(\mathcal{C})$ with the category of $\bar{Z}$-comodules, then a right adjoint of $U$ is the free comodule functor $V \mapsto (\bar{Z}(V), \delta_{V})$.
Equation \eqref{eq:coint-def} says that a cointegral is a morphism $R(\unitobj) \to \unitobj$ of $\bar{Z}$-comodules.
Hence, by~\eqref{eq:Fb-tr-uniq}, a non-zero cointegral in $\mathcal{C}$ exists and such a cointegral is unique up to scalar multiple.

It is easy to see that $\varepsilon_{\unitobj}: A \to \unitobj$ is a morphisms of algebras in $\mathcal{C}$. If we regard $\unitobj \in \mathcal{C}$ as a left $A$-module by $\varepsilon_{\unitobj}$, then an integral is precisely a morphism $\unitobj \to A$ of left $A$-modules. Since $A \cong A^*$ as right $A$-modules, we have
\begin{equation}
  \label{eq:Fb-integ-pf}
  \Hom_{A|-}(\unitobj, A)
  \cong \Hom_{-|A}(A^*, \unitobj^*)
  \cong \Hom_{-|A}(A, \unitobj)
  \cong \Hom_{\mathcal{C}}(\unitobj, \unitobj)
  \cong k,
\end{equation}
where $\Hom_{A|-}(X,Y)$ and $\Hom_{-|A}(X,Y)$ are the sets of morphisms of left and right $A$-modules, respectively.
This means that a non-zero integral in $\mathcal{C}$ exists and, moreover, such an integral is unique up to scalar multiple.

In the theory of finite-dimensional Hopf algebras, it is well-known that the pairing between cointegrals and integrals is non-degenerate.
To formulate an analogous fact in our setting, we consider a paring
\begin{equation}
  \langle \, , \, \rangle: \CF(\mathcal{C}) \times \CE(\mathcal{C}) \to k,
  \quad \langle f, a \rangle \, \id_{\unitobj} = f \circ a.
\end{equation}

Now we fix a non-zero cointegral $\lambda$ and define $\phi_{\lambda}$, $e_{\lambda}$ and $d_{\lambda}$ as in \S\ref{subsec:unimoftc}.

\begin{proposition}
  \label{prop:integ-non-degen}
  $\Lambda := (\id_A \otimes \varepsilon_{\unitobj}) \circ d_{\lambda}$ is an integral such that $\langle \lambda, \Lambda \rangle = 1$.
\end{proposition}

\begin{proof}
  The isomorphism $\theta: \Hom_{-|A}(A, \unitobj) \to \Hom_{A|-}(\unitobj, A)$ obtained by composing the inverses of the first two isomorphisms in \eqref{eq:Fb-integ-pf} is given by
  \begin{align*}
    \theta: \quad f \mapsto f \circ \phi_{\lambda}^{-1}
    \mapsto (\id_A \otimes (f \circ \phi_{\lambda}^{-1})) \circ \coev_A
    = (\id_A \otimes f) \circ d_{\lambda}
  \end{align*}
  for $f \in \Hom_{-|A}(A, \unitobj)$.
  Since $\varepsilon_{\unitobj} \in \Hom_{-|A}(A, \unitobj)$, $\Lambda = \theta(\varepsilon_{\unitobj})$ is an integral.
  To show $\langle \lambda, \Lambda \rangle = 1$, we recall that $A$ is a coalgebra with the comultiplication $\Delta$ given by~\eqref{eq:Fb-comul} and the counit $\lambda$. Hence,
  \begin{equation*}
    \lambda \circ \Lambda
    = (\lambda \otimes \varepsilon_{\unitobj}) \circ d_{\lambda}
    = (\lambda \otimes \varepsilon_{\unitobj}) \circ \Delta \circ u
    = \varepsilon_{\unitobj} \circ u = \id_{\unitobj}.
    \qedhere
  \end{equation*}
\end{proof}

The equivalence $(1) \Leftrightarrow (4)$ of the following proposition can be thought of as a generalization of the Maschke theorem on the semisimplicity of finite group algebras and finite-dimensional Hopf algebras.

\begin{proposition}
  \label{prop:integ-Maschke}
  Let $\Lambda \in \CE(\mathcal{C})$ be a non-zero integral in the unimodular finite tensor category $\mathcal{C}$. Then the following assertions are equivalent:
  \begin{enumerate}
  \item $\mathcal{C}$ is a semisimple abelian category.
  \item $\End(\id_{\mathcal{C}})$ is a semisimple algebra.
  \item The integral $\Lambda \in \CE(\mathcal{C})$ is not a nilpotent element.
  \item $\langle \varepsilon_{\unitobj}, \Lambda \rangle \ne 0$, where $\varepsilon$ is the counit of the central Hopf comonad on $\mathcal{C}$.
  \end{enumerate}
\end{proposition}
\begin{proof}
  If $\mathcal{C}$ is semisimple, then $\End(\id_{\mathcal{C}})$ is isomorphic to a finite direct product of the base field $k$. Thus (1) implies (2). The implication (2) $\Rightarrow$ (3) follows from the fact that $\CE(\mathcal{C})$ is isomorphic to the commutative algebra $\End(\id_{\mathcal{C}})$. Now suppose that $\Lambda$ is not nilpotent. Then we have $\langle \varepsilon_{\unitobj}, \Lambda \rangle \, \Lambda = \Lambda \cdot \Lambda \ne 0$ and thus $\langle \varepsilon_{\unitobj}, \Lambda \rangle \ne 0$. Hence (3) implies (4).

  Finally, we prove that (4) implies (1). Suppose $\langle \varepsilon_{\unitobj}, \Lambda \rangle \ne 0$.
  By normalizing the given integral, we may assume that $\langle \varepsilon_{\unitobj}, \Lambda \rangle = 1$ so that $\Lambda$ is an idempotent.
  Recall that $\mathcal{C}$ is semisimple if and only if the unit object $\unitobj \in \mathcal{C}$ is projective.
  Thus, to show that $\mathcal{C}$ is semisimple, it is enough to show that every epimorphism $p: X \to \unitobj$ in $\mathcal{C}$ splits.
  Let $e$ denote the natural transformation corresponding to $\Lambda$.
  Since $e$ is an idempotent, we have $X = T \oplus K$, where $T$ and $K$ are the image and the kernel of $e_X: X \to X$, respectively.
  By the definition of an integral, we have $\rho_T = \varepsilon_{\unitobj} \otimes \id_T$.
  Hence, by Proposition~\ref{prop:triv-obj}, $T$ is a trivial object.
  On the other hand, if $f: K \to \unitobj$ is a morphism in $\mathcal{C}$, then
  $f = e_{\unitobj} \circ f = f \circ e_K = 0$ by the naturality of $e$. Thus $p = \bar{p} \oplus 0: T \oplus K \to \unitobj$ for some $\bar{p}: T \to \unitobj$.
  Since $T$ is a trivial object, $\bar{p}$ splits. This implies that $p$ also splits.
\end{proof}

\subsection{Trace formula}

In the theory of finite-dimensional Hopf algebras, there are some trace formulas involving integrals.
Here we formulate and prove one of such formulas in the setting of unimodular finite tensor categories.

For simplicity, we assume that $R(V) = (\bar{Z}(V), \delta_V)$ is the free $\bar{Z}$-comodule functor
and set $\mathbf{A} = R(\unitobj)$ as in \S\ref{subsec:unimoftc}.
For $f \in \CF(\mathcal{C})$, we denote by $\tilde{f}: \mathbf{A} \to \mathbf{A}$ the morphism in $\mathcal{Z}(\mathcal{C})$ corresponding to $f$ via the isomorphism of Theorem~\ref{thm:class-ft-induc}.
Recall that the {\em Drinfeld isomorphism} in $\mathcal{Z}(\mathcal{C})$ is given by
\begin{equation*}
  \psi_{\mathbf{X}}: \mathbf{X}
  \xrightarrow{\  \id \otimes \coev \ }
  \mathbf{X} \otimes \mathbf{X}^* \otimes \mathbf{X}^{**}
  \xrightarrow{\  \sigma \otimes \id \ }
  \mathbf{X}^* \otimes \mathbf{X} \otimes \mathbf{X}^{**}
  \xrightarrow{\  \eval \otimes \id \ }
  \mathbf{X}^{**}
\end{equation*}
for $\mathbf{X} \in \mathcal{Z}(\mathcal{C})$, where $\sigma$ is the braiding of $\mathcal{Z}(\mathcal{C})$.
By \cite[\S7]{2013arXiv1309.4539S}, $\psi_{\mathbf{A}} = j_A$ if $\mathcal{C}$ has a pivotal structure $j$.
Thus, by abuse of notation, we define $\trace(\xi)$ for $\xi: A \to A$ in $\mathcal{C}$ to be the element of $k$ corresponding to the morphism
\begin{equation*}
  \unitobj
  \xrightarrow{\  \coev \ } A \otimes A^*
  \xrightarrow{\  \xi \otimes \id_A \ } A \otimes A^*
  \xrightarrow{\  \psi_{\mathbf{A}} \otimes \id_{A^*} \ } A^{**} \otimes A^{*}
  \xrightarrow{\  \eval \ } \unitobj
\end{equation*}
via the canonical isomorphism $k \cong \End_{\mathcal{C}}(\unitobj)$.

Now let $\lambda$ be a cointegral in $\mathcal{C}$, and let $\Lambda$ be the integral such that $\langle \lambda, \Lambda \rangle = 1$ (which exists by Proposition~\ref{prop:integ-non-degen}).
The following proposition can be thought of as an analogue of Radford's trace formula \cite[Proposition 2]{MR1265853}; see \S\ref{subsec:Hopf-alg-Fourier}.

\begin{proposition}
  \label{prop:trace-formula}
  $\trace(\tilde{f}) = \langle f, \Lambda \rangle \, \langle \lambda, u \rangle$ for all $f \in \CF(\mathcal{C})$.
\end{proposition}
\begin{proof}
  By the definition of the Drinfeld isomorphism, $\trace(\tilde{f})$ is equal to
  \begin{equation*}
    \unitobj
    \xrightarrow{\quad \coev \quad}
    \mathbf{A} \otimes \mathbf{A}^*
    \xrightarrow{\quad \tilde{f} \otimes \id \quad}
    \mathbf{A} \otimes \mathbf{A}^*
    \xrightarrow{\quad \sigma \quad}
    \mathbf{A}^* \otimes \mathbf{A}
    \xrightarrow{\quad \eval \quad} \unitobj.
  \end{equation*}
  This does not depend on the choice of a left dual object of $\mathbf{A}$.
  Thus we define $e_{\lambda}$ and $d_{\lambda}$ by \eqref{eq:Fb-dual-basis}
  and choose the triple $(\mathbf{A}, e_{\lambda}, d_{\lambda})$ as a left dual object of $\mathbf{A}$.
  Then, by~\eqref{eq:adj-alg-comm}, we have
  \begin{equation}
    \label{eq:trace-formula-pf-1}
    \trace(\tilde{f})
    = e_{\lambda} \circ \sigma_{\mathbf{A},\mathbf{A}} \circ (\tilde{f} \otimes \id_\mathbf{A}) \circ d_{\lambda}
    = e_{\lambda} \circ (\tilde{f} \otimes \id_\mathbf{A}) \circ d_{\lambda}
  \end{equation}
  for all $f \in \CF(\mathcal{C})$.

  Since $d_{\lambda}: \unitobj \to \mathbf{A} \otimes \mathbf{A}$ is a morphism in $\mathcal{Z}(\mathcal{C})$, we have
  \begin{equation}
    \label{eq:lem-trace-1-1}
    \bar{Z}(d_{\lambda}) \circ \bar{Z}_0
    = \bar{Z}_2(A, A) \circ (\delta_{\unitobj} \otimes \delta_{\unitobj}) \circ d_{\lambda}
  \end{equation}
  by the definition of the tensor product of $\bar{Z}$-comodules. We also have
  \begin{equation}
    \label{eq:lem-trace-1-2}
    \bar{Z}(\varepsilon_{\unitobj}) \circ \delta_{\unitobj} = \id_{\bar{Z}(\unitobj)}
  \end{equation}
  by the definition of a comonad. Now we compute
  \begin{align*}
    \allowdisplaybreaks
    & m \circ (\tilde{f} \otimes \id_A) \circ d_{\lambda} \\
    & = \bar{Z}_2(\unitobj, \unitobj) \circ (\bar{Z}(f) \otimes \bar{Z}(\varepsilon_{\unitobj})) \circ (\delta_{\unitobj} \otimes \delta_{\unitobj}) \circ d_{\lambda}
    & & \text{(by~\eqref{eq:class-ft-adj-iso} and~\eqref{eq:lem-trace-1-2})} \\
    & = \bar{Z}(f \otimes \varepsilon_{\unitobj}) \circ \bar{Z}_2(A, A) \circ (\delta_{\unitobj} \otimes \delta_{\unitobj}) \circ d_{\lambda}
    & & \text{(by the naturality of $\bar{Z}_2$)} \\
    & = \bar{Z}(f \otimes \varepsilon_{\unitobj}) \circ \bar{Z}(d_{\lambda}) \circ \bar{Z}_0
    & & \text{(by~\eqref{eq:lem-trace-1-1})} \\
    & = \bar{Z}(f \circ \Lambda) \circ \bar{Z}_0
    & & \text{(by Proposition~\ref{prop:integ-non-degen})} \\
    & = \langle f, \Lambda \rangle \, \bar{Z}(\id_{\unitobj}) \circ \bar{Z}_0
    = \langle f, \Lambda \rangle \, u.
  \end{align*}
  The claim now follows from~\eqref{eq:trace-formula-pf-1} and the definition of $e_{\lambda}$.
\end{proof}

\subsection{Fourier transform}

Let $\lambda$ be a non-zero cointegral in $\mathcal{C}$, and let $\Lambda$ be the integral such that $\langle \lambda, \Lambda \rangle = 1$.

\begin{definition}
  The {\em Fourier transform} (associated to $\lambda$) is the linear map
  \begin{equation}
    \label{eq:Fourier-tr}
    \mathfrak{F}_{\lambda}: \CE(\mathcal{C}) \to \CF(\mathcal{C}),
    \quad \mathfrak{F}_{\lambda}(a) = \lambda \leftharpoonup \mathfrak{S}(a),
  \end{equation}
  where $\mathfrak{S}: \CE(\mathcal{C}) \to \CE(\mathcal{C})$ is the antipodal operator of Definition~\ref{def:antipodal}
  and $\leftharpoonup$ is the action of $\CE(\mathcal{C})$ on $\CF(\mathcal{C})$ given by
  \begin{equation}
    f \leftharpoonup b = f \circ m \circ (b \otimes \id_A)
    \quad (f \in \CF(\mathcal{C}), b \in \CE(\mathcal{C})).
  \end{equation}
\end{definition}

The Fourier transform should be defined in a form involving $\mathfrak{S}$ in view of the case of finite-dimensional Hopf algebras (see the discussion in \S\ref{subsec:Hopf-alg-Fourier}).
However, it is convenient to consider the `twisted' Fourier transform
\begin{equation}
  \label{eq:tw-Fourier-tr}
  \mathfrak{F}'_{\lambda} = \mathfrak{F}_{\lambda} \circ \mathfrak{S}^{-1}.
\end{equation}
Since $(A, e_{\lambda}, d_{\lambda})$ is a left dual object of $A$, the inverse of $\mathfrak{F}'_{\lambda}$ is given by
\begin{equation}
  \label{eq:tw-Fourier-tr-inv}
  (\mathfrak{F}'_{\lambda})^{-1}(f) = (f \otimes \id_A) \circ d_{\lambda}
\end{equation}
for $f \in \CF(\mathcal{C})$. Thus $\mathfrak{F}_{\lambda}$ is also invertible.

As an application of the (twisted) Fourier transform, we prove:

\begin{theorem}
  \label{thm:unimo-gr}
  For a unimodular finite tensor category $\mathcal{C}$ admitting a pivotal structure, the following assertions are equivalent:
  \begin{enumerate}
  \item $\mathcal{C}$ is semisimple.
  \item $\dim_k \Gr_k(\mathcal{C}) = \dim_k \CF(\mathcal{C})$.
  \end{enumerate}
  Thus, the map $\mathsf{ch}: \Gr_k(\mathcal{C}) \to \CF(\mathcal{C})$ is bijective if and only if $\mathcal{C}$ is semisimple.
\end{theorem}
\begin{proof}
  For simplicity, we set $\mathsf{E} = \CE(\mathcal{C})$ and $\mathsf{F} = \CF(\mathcal{C})$.
  Let $\psi: \mathsf{E} \to \End(\id_{\mathcal{C}})$ be the isomorphism given by~\eqref{eq:CE-and-End-id}. 
  By Schur's lemma, for each simple object $V \in \mathcal{C}$, there is a linear map $\xi_V: \mathsf{E} \to k$ such that $\psi(a)_V = \xi_V(a) \cdot \id_V$ for all $a \in \mathsf{E}$.
  One can easily verify that $\xi_V$ is an algebra map and
  \begin{equation*}
    \mathsf{ch}(V) \leftharpoonup a = \xi_V(a) \cdot \mathsf{ch}(V)
  \end{equation*}
  for all $a \in \mathsf{E}$. This means that the one-dimensional space spanned by an irreducible character is a one-dimensional submodule of the $\mathsf{E}$-module $\mathsf{F}$.

  The implication (1) $\Rightarrow$ (2) has been proved in Example~\ref{ex:indep-irr-ch-fusion}.
  If, conversely, (2) holds, then $\mathsf{F}$ is spanned by the irreducible characters by Theorem~\ref{thm:irr-ch-indep}, and hence a semisimple $\mathsf{E}$-module by the above argument.
  Since $\mathfrak{F}'_{\lambda}: \mathsf{E} \to \mathsf{F}$ is an isomorphism of $\mathsf{E}$-modules, $\mathsf{E}$ is also a semisimple $\mathsf{E}$-module. By Proposition~\ref{prop:integ-Maschke}, $\mathcal{C}$ is semisimple.
\end{proof}

We note that the unimodularity assumption of the above theorem is necessarily, since the equality of (2) holds if $\mathcal{C}$ is the representation category of the Taft algebra, which is neither unimodular nor semisimple; see Example~\ref{ex:Taft}.

The implication (2) $\Rightarrow$ (1) of the theorem can be rephrased as that, if $\mathcal{C}$ is a non-semisimple unimodular pivotal finite tensor category, then there exists an element of $\CF(\mathcal{C})$ that is not a linear combination of internal characters.
The gap between the dimensions of $\CF(\mathcal{C})$ and $\Gr_k(\mathcal{C})$ is an interesting subject but will not be discussed in this paper.

\subsection{The case of Hopf algebras}
\label{subsec:Hopf-alg-Fourier}

To conclude this section, we consider the case where $\mathcal{C} = H\mbox{-{\sf mod}}$ is the representation category a finite-dimensional Hopf algebra $H$ over $k$.
We use the same notations as in \S\ref{subsec:Hopf-alg} and identify $\mathcal{Z}(\mathcal{C})$ with the Yetter-Drinfeld category.

As we have seen in \S\ref{subsec:Hopf-alg}, the algebra $\CE(\mathcal{C})$ can be identified with the center of the algebra $H$.
By Definition~\ref{def:integ}, an integral in $\mathcal{C}$ is an element $\Lambda \in H$ such that
\begin{equation*}
  h_{(1)} \Lambda S(h_{(2)}) = \varepsilon(h) \Lambda
  \quad \text{and} \quad
  a \Lambda = \varepsilon(a) \Lambda
\end{equation*}
for all $a, h \in H$. The former equation implies that the element $\Lambda$ is central,
and the latter means that $\Lambda$ is a left integral in $H$ in the usual sense.
Thus a non-zero integral in $\mathcal{C}$ exists if and only if $H$ is unimodular.

By definition, a cointegral in $\mathcal{C}$ is a morphism $\lambda: R(k) \to k$ of Yetter-Drinfeld modules.
Thus it is the same thing as a linear map $\lambda: H \to k$ such that
\begin{equation*}
  a_{(1)} \, \langle \lambda, a_{(2)} \rangle
  = \varepsilon(a) 1_H
  \quad \text{and} \quad
  \langle \lambda, h_{(1)} a S(h_{(2)}) \rangle = \langle \varepsilon, h \rangle \langle \lambda, a \rangle
\end{equation*}
for all $a, h \in H$. The former condition means that $\lambda$ is a right cointegral on $H$.
By the result of Radford \cite{MR1265853}, a non-zero right cointegral on $H$ satisfies the latter condition if and only if $H$ is unimodular.

Now we suppose that $H$ is unimodular.
We fix a non-zero cointegral $\lambda$ in $\mathcal{C}$, and then let $\Lambda$ be the (two-sided) integral such that $\langle \lambda, \Lambda \rangle = 1$.
The Frobenius property of the algebra $\mathbf{A} := R(k)$ in ${}^H_H \mathcal{YD}$ has been pointed out by Ishii and Masuoka in \cite{MR3265394} in their study of handlebody-knot invariants.
According to them, the map
\begin{equation}
  d_{\lambda}: k \to \mathbf{A} \otimes \mathbf{A},
  \quad 1 \mapsto \Lambda_{(2)} \otimes S^{-1}(\Lambda_{(1)})
\end{equation}
is a morphism in ${}^H_H \mathcal{YD}$ and the triple $(\mathbf{A}, e_{\lambda}, d_{\lambda})$, where $e_{\lambda} = \lambda \circ m$, is a left dual object of $\mathbf{A}$.
The formula of $\Lambda$ in Proposition~\ref{prop:integ-non-degen} reduces to
\begin{equation*}
  \Lambda = \Lambda_{(2)} \langle \varepsilon, S^{-1}(\Lambda_{(1)}) \rangle,
\end{equation*}
which is obviously true from the axiom of Hopf algebras.

Let $\psi$ be the Drinfeld isomorphism in $\mathcal{Z}(\mathcal{C})$. Then we have
\begin{equation*}
  \langle \psi_{\mathbf{A}}(a), p \rangle = \langle p, S^2(a) \rangle
  \quad (a \in \mathbf{A}, p \in H^*)
\end{equation*}
(see the proof of \cite[Proposition 7.1]{2013arXiv1309.4539S}).
For $f \in \CF(\mathcal{C})$, the morphism $\widetilde{f}: \mathbf{A} \to \mathbf{A}$ of Proposition~\ref{prop:trace-formula} is the linear map given
by $a \mapsto a_{(1)} \langle f, a_{(2)} \rangle$ for $a \in \mathbf{A}$.
Hence, in our context, the claim of the proposition reduces to
\begin{equation*}
  \Trace\Big( A \to A; \, a \mapsto S^{2}(a_{(1)}) \langle f, a_{(2)} \rangle \Big) = \langle f, \Lambda \rangle \langle \lambda, 1 \rangle,
\end{equation*}
which is a special case of Radford's formula given in \cite{MR1265853}.

Recall that $H$ acts on $H^*$ from the right by $\langle f \leftharpoonup h, h' \rangle = \langle f, h h' \rangle$.
If we identify $\CE(\mathcal{C})$ with the center of $H$, then the Fourier transform associated $\lambda$ is given by
\begin{equation}
  \label{eq:Fourier-Hopf}
  \mathfrak{F}_{\lambda}: \CE(\mathcal{C}) \to \CF(\mathcal{C}),
  \quad \mathfrak{F}_{\lambda}(a) = \lambda \leftharpoonup S^{-1}(a)
\end{equation}
for $a \in \CE(\mathcal{C})$.
This operator coincides with a restriction of the Fourier transform on $H^{\mathrm{cop}}$ introduced and studied by Cohen and Westreich in \cite{MR2349620}.
We note that the inverse of \eqref{eq:Fourier-Hopf} is given by
\begin{equation}
  \label{eq:Fourier-Hopf-inv}
  \mathfrak{F}_{\lambda}^{-1}(f) = \langle f, \Lambda_{(1)} \rangle \Lambda_{(2)}
  \quad (f \in \CF(\mathcal{C})).
\end{equation}
The space of symmetric forms on a finite-dimensional Hopf algebra is also studied by Cohen and Westreich in \cite{MR2464107}.
Our main result in this section (Theorem~\ref{thm:unimo-gr}) corresponds to a part of \cite[Corollary 2.3]{MR2464107}.

\section{Applications to fusion categories}
\label{sec:appl-fusion}

\subsection{Integrals in a fusion category}

Finally, we give some applications of our theory to fusion categories.
Throughout this section, $\mathcal{C}$ is a pivotal fusion category over an algebraically closed field $k$,
and we use the set of notations as in \S\ref{subsec:FTC}.
Furthermore, let $\{ V_0, \dotsc, V_m \}$ be a complete set of representatives of isomorphism classes of simple objects with $V_0 = \unitobj$.
For $i \in \{ 0, \dotsc, m \}$, we define $i^* \in \{ 0, \dotsc, m \}$ by $V_{i^*} \cong V_i^*$.
Then $i \mapsto i^*$ is an involution on $\{ 0, \dotsc, m \}$.

Since a  fusion category is unimodular \cite{MR2097289},
$\mathcal{C}$ has a non-zero integral and a non-zero cointegral in the sense of Definition~\ref{def:integ}.
To describe them, we introduce further notations.
As we have remarked in Example~\ref{ex:indep-irr-ch-fusion}, we may choose
\begin{equation*}
  A = \bigoplus_{i = 0}^m V_i \otimes V_i^*
\end{equation*}
as a realization of the adjoint algebra. For $i = 0, \dotsc, m$, we define $e_i$ and $\chi_i$ by
\begin{equation*}
  e_i: \unitobj \xrightarrow{\ \coev \ } V_i \otimes V_i^*
  \xrightarrow{\ \text{inclusion} \ } A
  \quad \text{and} \quad
  \chi_i: A \xrightarrow{\ \text{projection} \ } V_i \otimes V_i^*
  \xrightarrow{\ \widetilde{\eval} \ } \unitobj,
\end{equation*}
respectively.
It is easy to see that the sets $\{ \chi_i \}_{i = 0, \dotsc, m}$ and $\{ e_i \}_{i - 0, \dotsc, m}$ are bases of $\CF(\mathcal{C})$ and $\CE(\mathcal{C})$, respectively, such that
\begin{equation}
  \label{eq:fusion-cf-ce-paring}
  \langle \chi_i, e_j \rangle = \dim(V_j) \Kdelta_{i, j}
  \quad (i, j  = 0, \dotsc, m),
\end{equation}
where $\Kdelta_{-,-}$ is the Kronecker delta.

As we have seen in Example~\ref{ex:indep-irr-ch-fusion}, $\chi_i = \mathsf{ch}(V_i)$ for $i = 0, \dotsc, m$.
The central element $e_i$ corresponds via~\eqref{eq:CE-and-End-id} to the endomorphism on $\id_{\mathcal{C}}$
that is the identity on $V_i$ and zero on the other $V_j$'s.
Thus we have
\begin{equation}
  \label{eq:fusion-cent-elem}
  e_i \cdot e_j = \Kdelta_{i,j} e_i
  \quad \text{and} \quad
  \mathfrak{S}(e_i) = e_{i^*} \quad (i, j = 0, \dotsc, m),
\end{equation}
where $\mathfrak{S}: \CE(\mathcal{C}) \to \CE(\mathcal{C})$ is the antipodal operator of Definition~\ref{def:antipodal}.

\begin{lemma}
  $e_0 \in \CE(\mathcal{C})$ is an integral in $\mathcal{C}$.
\end{lemma}
\begin{proof}
  The linear map $\epsilon: \CE(\mathcal{C}) \to k$ given by $a \mapsto \langle \varepsilon_{\unitobj}, a \rangle$ is a homomorphism of algebras.
  If $\Lambda$ is an integral in $\mathcal{C}$, then we have
  \begin{equation}
    \label{eq:fusion-integ-cond}
    a \cdot \Lambda = \epsilon(a) \Lambda \quad (a \in \CE(\mathcal{C})).
  \end{equation}
  by the definition of an integral.
  By \eqref{eq:fusion-cf-ce-paring} and \eqref{eq:Fb-alg-alpha-integ}, $\Lambda = e_0$ satisfies~\eqref{eq:fusion-integ-cond}.
  On the other hand, since $\CE(\mathcal{C}) \cong k \times \dotsb \times k$ as algebras, an element $\Lambda$ satisfying \eqref{eq:fusion-integ-cond} is unique up to scalar multiples.
  Thus we conclude that $e_0$ is an integral in $\mathcal{C}$.
\end{proof}

To give a description of a cointegral, we provide:

\begin{lemma}
  $\Gr_k(\mathcal{C})$ is a symmetric Frobenius algebra with the trace given by
  \begin{equation*}
    \Gr_k(\mathcal{C}) \to k,
    \quad [X] \mapsto \dim_k \Hom_{\mathcal{C}}(\unitobj, X).
  \end{equation*}
\end{lemma}
\begin{proof}
  Let $\tau$ denote the above linear map.
  The well-definedness of $\tau$ follows from the semisimplicity of $\mathcal{C}$.
  By the rigidity, we have
  \begin{equation*}
    \tau([V_i] \cdot [V_j])
    = \dim_k \Hom_{\mathcal{C}}(\unitobj, V_i \otimes V_j)
    = \dim_k \Hom_{\mathcal{C}}(V_i^*, V_j)
    = \Kdelta_{i^*, j}
  \end{equation*}
  for $i, j = 0, \dotsc, m$.
  Thus $\tau$ is non-degenerate.
  The symmetry of $\tau$ is obvious from the above computation.
\end{proof}

Recall from Example~\ref{ex:indep-irr-ch-fusion} that $\CF(\mathcal{C})$ is isomorphic to $\Gr_k(\mathcal{C})$ as an algebra.
Thus $\CF(\mathcal{C})$ is a symmetric Frobenius algebra with the trace $\tau$ given by
\begin{equation}
  \label{eq:fusion-cf-trace}
  \tau(\chi_i) = \Kdelta_{i,0} \quad (i = 0, \dotsc, m).
\end{equation}
By the above proof, $\tau(\chi_i \star \chi_j) = \Kdelta_{i, j^*}$. Thus $E := \sum_{i = 0}^m \chi_{i^*} \otimes \chi_i \in \CF(\mathcal{C})^{\otimes 2}$ is an element such that the linear map
\begin{equation}
  \label{eq:Fb-alg-iso}
  \Theta: \Hom_k(\CF(\mathcal{C}), k) \to \CF(\mathcal{C}),
  \quad \alpha \mapsto (\alpha \otimes \id)(E)
\end{equation}
is bijective. Let, in general, $\alpha: \CF(\mathcal{C}) \to k$ be an algebra map.
By the basic theory of Frobenius algebras, we see that $\lambda = \Theta(\alpha)$ satisfies
\begin{equation}
  \label{eq:Fb-alg-alpha-integ}
  f \star \lambda = \alpha(f) \, \lambda
\end{equation}
for all $f \in \CF(\mathcal{C})$, and an element $\lambda \in \CF(\mathcal{C})$ satisfying~\eqref{eq:Fb-alg-alpha-integ} is unique up to scalar multiple.
Moreover, the symmetricity of $\tau$ implies that $\Theta(\alpha) \in \CF(\mathcal{C})$ is central.

Now we consider the map $d: \CF(\mathcal{C}) \to k$ defined by $f \mapsto \langle f, u \rangle$.
It is easy to see that $d$ is an algebra map sending the internal character of $X \in \mathcal{C}$ to the pivotal dimension of $X$.
Applying the above argument to $\alpha = d$ gives the following description of cointegrals in $\mathcal{C}$.

\begin{lemma}
  \label{lem:fusion-coint}
  A class function $\lambda \in \CF(\mathcal{C})$ is a cointegral if and only if
  \begin{equation*}
    f \star \lambda = \langle f, u \rangle \, \lambda    
  \end{equation*}
  for all $f \in \CF(\mathcal{C})$. Thus, by the above argument, we see that
  \begin{equation*}
    \Theta(d) = \sum_{i = 0}^m \dim(V_i^*) \, \chi_i \in \CF(\mathcal{C})
  \end{equation*}
  is a cointegral in $\mathcal{C}$ and belongs to the center of $\CF(\mathcal{C})$.
\end{lemma}
\begin{proof}
  Let $I$ be the space of class functions $\lambda$ satisfying~\eqref{eq:Fb-alg-alpha-integ} with $\alpha = d$.
  By the above argument, $I$ is one-dimensional.
  If $\lambda \in \CF(\mathcal{C})$ is a cointegral in $\mathcal{C}$, then, for all $f \in \CF(\mathcal{C})$, we have
  \begin{equation*}
    f \star \lambda
    = f \circ \bar{Z}(\lambda) \circ \Kdelta_{\unitobj}
    = f \circ u \circ \lambda
    = \alpha(f) \cdot \lambda
  \end{equation*}
  by~\eqref{eq:class-ft-prod} and~\eqref{eq:coint-def}.
  Thus the space $I'$ of cointegrals is a subspace of $I$.
  Since $I$ is one-dimensional, and since a non-zero cointegral in $\mathcal{C}$ exists, we have $I = I'$.
  The proof is done.
\end{proof}

\subsection{Properties of the Grothendieck algebra}

The dimension of $\mathcal{C}$ is given by
\begin{equation*}
  \dim(\mathcal{C}) = \sum_{i = 0}^m \dim(V_i^*) \dim(V_i) \in k.
\end{equation*}
It is known that $\dim(\mathcal{C}) \ne 0$ if $\mathrm{char}(k) = 0$ \cite{MR2183279}.
On the other hand, $\dim(\mathcal{C})$ happens to be zero in the case where $\mathrm{char}(k) > 0$.

\begin{definition}
   $\mathcal{C}$ is said to be {\em non-degenerate} if $\dim(\mathcal{C}) \ne 0$ in $k$.
\end{definition}

We now prove the following theorem:

\begin{theorem}
  \label{thm:Gro-alg-ss}
  For the pivotal fusion category $\mathcal{C}$, the following four conditions are equivalent:
  \begin{itemize}
  \item [(1)] $\Gr_k(\mathcal{C})$ is a semisimple algebra.
  \item [(2)] $\mathcal{C}$ is non-degenerate.
  \item [(3)] $\mathcal{Z}(\mathcal{C})$ is a semisimple abelian category.
  \item [(4)] $R(\unitobj) \in \mathcal{Z}(\mathcal{C})$ is a semisimple object.
  \end{itemize}
\end{theorem}

Note that M\"uger \cite{MR1966525} proved (2) $\Rightarrow$ (3). Brugui\`eres and Virelizier \cite{MR3079759} proved (2) $\Leftrightarrow$ (3) in a more general setting than ours.
Ostrik proved (1) $\Rightarrow$ (2) and conjectured that the converse holds in \cite{2015arXiv150301492O}.
Highlighting the role of cointegrals, we also give another proof of (1) $\Rightarrow$ (2).

\begin{proof}
  As we have remarked, (2) $\Leftrightarrow$ (3) has been known. (3) $\Rightarrow$ (4) is obvious.
  To show (4) $\Rightarrow$ (1), we recall from  Example~\ref{ex:indep-irr-ch-fusion} that there are isomorphisms
  \begin{equation}
    \label{eq:fusion-Gr-iso}
    \Gr_k(\mathcal{C}) \cong \CF(\mathcal{C}) \cong \End_{\mathcal{Z}(\mathcal{C})}(R(\unitobj))
  \end{equation}
  of algebras.
  Thus, if $R(\unitobj)$ is a semisimple object, then $\End_{\mathcal{Z}(\mathcal{C})}(R(\unitobj))$ is semisimple, and therefore so is $\Gr_k(\mathcal{C})$.

  Finally, we prove (1) $\Rightarrow$ (2).
  Suppose that (1) holds. Then $\CF(\mathcal{C})$ has no non-zero central nilpotent elements, since it is isomorphic to a finite product of matrix algebras.
  Now let $\lambda$ be the cointegral in $\mathcal{C}$ given by Lemma~\ref{lem:fusion-coint}. Since $\lambda$ is a non-zero central element, we have
  $\dim(\mathcal{C}) \, \lambda = \lambda \star \lambda \ne 0$.
  Hence $\dim(\mathcal{C}) \ne 0$.
\end{proof}

For a pair $(H, K)$ consisting of a semisimple Hopf algebra $H$ over an algebraically closed field of characteristic zero and its Hopf subalgebra $K$, Andruskiewitsch and Natale \cite{MR1812940} defined the {\em Hecke algebra} $\mathcal{H}(H, K)$. The pair $(H, K)$ is called a {\em Gelfand pair} if the algebra $\mathcal{H}(H, K)$ is commutative. As shown in \cite{MR1812940}, this is equivalent to that the left $H$-module $H \otimes_K \unitobj_K$ is multiplicity-free, where $\unitobj_K$ is the trivial $K$-module. They also showed that $(D(H), H)$ is a Gelfand pair if and only if the Grothendieck ring of $H\mbox{-{\sf mod}}$ is commutative, where $D(H)$ is the Drinfeld double of $H$. The following result can be thought of as a category-theoretical generalization of their result:

\begin{theorem}
  \label{thm:Gro-alg-comm}
  For a non-degenerate pivotal fusion category $\mathcal{C}$, the following assertions are equivalent:
  \begin{enumerate}
  \item $R(\unitobj) \in \mathcal{Z}(\mathcal{C})$ is multiplicity-free.
  \item The algebra $\Gr_k(\mathcal{C})$ is commutative.
  \end{enumerate}
\end{theorem}

For the case of $\mathrm{char}(k) = 0$, this theorem has been proved by Ostrik \cite{2013arXiv1309.4822O}.

\begin{proof}
  Under the assumption, $R(\unitobj) \in \mathcal{Z}(\mathcal{C})$ is semisimple.
  Thus (1) is equivalent to that the endomorphism algebra of $R(\unitobj)$ is commutative.
  Now the result follows from \eqref{eq:fusion-Gr-iso} in the proof of the previous theorem.
\end{proof}

\subsection{The character table}

We now suppose that $\mathcal{C}$ is a non-degenerate pivotal fusion category such that $\Gr_k(\mathcal{C})$ is commutative (although some of below definitions make sense without such assumptions).
If this is the case, $R(\unitobj) \in \mathcal{Z}(\mathcal{C})$ is multiplicity-free by Theorem~\ref{thm:Gro-alg-comm}.

\begin{definition}
  \label{def:conj-class}
  A {\em conjugacy class} of $\mathcal{C}$ is a simple direct summand of $R(\unitobj) \in \mathcal{Z}(\mathcal{C})$. For a given conjugacy class $\mathfrak{C}$, we define its {\em size} by
  \begin{equation*}
    |\mathfrak{C}| = \dim(\mathfrak{C}) \in k.
  \end{equation*}
  The {\em class sum} $\overline{\mathfrak{C}} \in \CE(\mathcal{C})$ is defined as follows:
  Let $e$ be the idempotent on $R(\unitobj)$ whose image is $\mathfrak{C}$.
  Then $\overline{\mathfrak{C}}$ is the element corresponding to $e$ via
  \begin{equation*}
    \End_{\mathcal{Z}(\mathcal{C})}(R(\unitobj))
    \xrightarrow{\quad \text{Theorem~\ref{thm:class-ft-induc}} \quad}
    \CF(\mathcal{C})
    \xrightarrow{\quad \mathfrak{F}_{\lambda}^{-1} \quad}
    \CE(\mathcal{C}),
  \end{equation*}
  where $\lambda$ is the cointegral in $\mathcal{C}$ such that $\langle \lambda, u \rangle = 1$.
\end{definition}

By Lemma~\ref{lem:fusion-coint}, such a cointegral is given by
\begin{equation}
  \label{eq:fusion-coint-normalized}
  \lambda = \frac{1}{\dim(\mathcal{C})} \sum_{i = 0}^m \dim(V_i^*) \chi_i.
\end{equation}

By Theorems~\ref{thm:Gro-alg-ss} and~\ref{thm:Gro-alg-comm}, the number of conjugacy classes of $\mathcal{C}$ is equal to the number of isomorphism classes of simple objects of $\mathcal{C}$.
Thus let $\mathfrak{C}_0, \dotsc, \mathfrak{C}_m$ be the conjugacy classes of $\mathcal{C}$.
Since the unit object $\unitobj_{\mathcal{Z}(\mathcal{C})} \in \mathcal{Z}(\mathcal{C})$ is always a subobject of $R(\unitobj)$,
we can assume $\mathfrak{C}_0 = \unitobj_{\mathcal{Z}(\mathcal{C})}$.

We note that the size of $\mathfrak{C}_j$ is non-zero in $k$ since it is the pivotal dimension of a simple object in a pivotal fusion category $\mathcal{Z}(\mathcal{C})$. Thus we set
\begin{equation}
  \label{eq:conj-class-aver}
  g_r := |\mathfrak{C}_r|^{-1} \cdot \overline{\mathfrak{C}}_r \in \CE(\mathcal{C})
\end{equation}
for $r = 0, \dotsc, m$. The element $g_r$ can be thought of as the ``average'' of the elements in the conjugacy classes.
Thus we define:

\begin{definition}
  \label{def:ct}
  We call the matrix $(\langle \chi_i, g_r \rangle)_{i, r = 0, \dotsc, m}$ the {\em character table} of $\mathcal{C}$.
\end{definition}

For $r = 0, \dotsc, m$, let $\widetilde{f}_r \in \End_{\mathcal{Z}(\mathcal{C})} (R(\unitobj))$ denote the idempotent whose image is $\mathfrak{C}_r$, and let $f_r \in \CF(\mathcal{C})$ be the corresponding class function.
The set $\{ f_r \}_{r = 0, \dotsc, m}$ is a complete set of primitive orthogonal idempotents of $\CF(\mathcal{C})$.
$f_r$'s are expressed by irreducible characters as follows:

\begin{theorem}
  \label{thm;ct-idempo}
  $\displaystyle f_r = \frac{|\mathfrak{C}_r|}{\dim(\mathcal{C})} \sum_{i = 0}^m \, \langle \chi_{i^*}, g_r \rangle \, \chi_i$ for all $r = 0, \dotsc, m$.
\end{theorem}
\begin{proof}
  We represent $\mathfrak{F}_{\lambda}$ and its inverse with respect to bases $\{ e_i \}$ and $\{ \chi_i \}$.
  By~\eqref{eq:fusion-cent-elem}, \eqref{eq:fusion-Gr-iso} and the definition of the Fourier transform, we have
  \begin{equation*}
    \langle \mathfrak{F}_{\lambda}(e_i), e_j \rangle
    = \frac{1}{\dim(\mathcal{C})} \sum_{\ell = 0}^m \dim(V_{\ell}^*) \langle \chi_{\ell}, \mathfrak{S}(e_i) e_j \rangle
    = \frac{\dim(V_{i}) \dim(V_i^*)}{\dim(\mathcal{C})} \Kdelta_{i^*, j}
  \end{equation*}
  for all $i, j = 0, \dotsc, m$. Thus, by~\eqref{eq:fusion-cf-ce-paring},
  \begin{equation}
    \label{eq:fusion-fourier}
    \mathfrak{F}_{\lambda}(e_i) = \frac{\dim(V_i)}{\dim(\mathcal{C})} \, \chi_{i^*}
    \quad \text{and} \quad
    \mathfrak{F}_{\lambda}^{-1}(\chi_i) = \frac{\dim(\mathcal{C})}{\dim(V_i)} \, e_{i^*}
    \quad (i = 0, \dotsc, m).
  \end{equation}
  Now we write $f_r = \sum_{i = 0}^m a_{i r} \chi_i$ ($a_{i r} \in k$). Then, by~\eqref{eq:fusion-cf-ce-paring} and \eqref{eq:fusion-fourier},
  \begin{equation*}
    \frac{|\mathfrak{C}_r|}{\dim(\mathcal{C})} \langle \chi_{i^*}, g_r \rangle
    = \frac{1}{\dim(\mathcal{C})} \langle \chi_{i^*}, \mathfrak{F}^{-1}_{\lambda}(f_r) \rangle
    = \sum_{j = 0}^m \frac{a_{j r}}{\dim(V_i)} \langle \chi_{i^*}, e_{j^*} \rangle
    = a_{i r}. \qedhere
  \end{equation*}
\end{proof}

The orthogonality relations can be derived from this theorem.
As a preparation, we show that every entry of the $0$-th row of the character table is $1$.

\begin{lemma}
  \label{lem:ct-0-th-row}
  $\langle \chi_0, g_r \rangle = 1$ for $r = 0, \dotsc, m$.
\end{lemma}
\begin{proof}
  For $a \in \CE(\mathcal{C})$, one easily sees that $\langle \chi_0, \mathfrak{S}^{-1}(a) \rangle = \langle \chi_0, a \rangle$.
  By~\eqref{eq:tw-Fourier-tr},  \eqref{eq:tw-Fourier-tr-inv} and Proposition~\ref{prop:integ-non-degen}, we have
  \begin{equation*}
    \langle \varepsilon_{\unitobj}, \mathfrak{F}_{\lambda}^{-1}(f_r) \rangle
    = \langle \varepsilon_{\unitobj}, (f_r \otimes \id_A) \circ d_{\lambda} \rangle
    = (f_r \otimes \varepsilon_{\unitobj}) \circ d_{\lambda}
    = \langle f_r, \Lambda \rangle
    = \trace(\widetilde{f}_r) = |\mathfrak{C}_r|.
  \end{equation*}
  Thus, $\langle \chi_0, g_r \rangle = |\mathfrak{C}_r|^{-1} \langle \varepsilon_{\unitobj}, \mathfrak{F}_{\lambda}^{-1}(f_r) \rangle = 1$.
\end{proof}

\begin{corollary}
  \label{cor:ct-ortho}
  We have the first orthogonality relation
  \begin{equation}
    \label{eq:ct-1st-ortho}
    \sum_{r = 0}^m \frac{|\mathfrak{C}_r|}{\dim(\mathcal{C})}
    \, \langle \chi_i, g_r \rangle \, \langle \chi_{j^*}, g_r \rangle = \Kdelta_{i, j}
    \quad (i, j = 0, \dotsc, m)
  \end{equation}
  and the second orthogonality relation
  \begin{equation}
    \label{eq:ct-2nd-ortho}
    \sum_{i = 0}^{m} \frac{|\mathfrak{C}_r|}{\dim(\mathcal{C})}
    \, \langle \chi_{i}, g_r \rangle \, \langle \chi_{i^*}, g_s \rangle = \Kdelta_{r,s}
    \quad (r, s = 0, \dotsc, m).
  \end{equation}
  For $i = 0, \dotsc, m$, we have
  \begin{equation}
    \label{eq:ct-chg-basis-2}
    \chi_i = \sum_{r = 0}^m \langle \chi_{i}, g_r \rangle \, f_r.
  \end{equation}
\end{corollary}
\begin{proof}
  We first prove the second orthogonality relation. By Theorem~\ref{thm;ct-idempo},
  \begin{equation*}
    \frac{\dim(\mathcal{C})}{|\mathfrak{C}_s|} \Kdelta_{r,s} f_s
    = \frac{\dim(\mathcal{C})}{|\mathfrak{C}_s|} f_r \star f_s
    = \frac{|\mathfrak{C}_r|}{\dim(\mathcal{C})} \sum_{i, j = 0}^m
    \, \langle \chi_{i}, g_r \rangle \, \langle \chi_{j}, g_s \rangle \, \chi_i \star \chi_j
  \end{equation*}
  for all $r, s = 0, \dotsc, m$. Applying~\eqref{eq:fusion-cf-trace} to both sides, we obtain
  \begin{equation*}
    \Kdelta_{r,s}  \langle \chi_0, g_r \rangle
    = \frac{|\mathfrak{C}_r|}{\dim(\mathcal{C})} \sum_{i = 0}^m \langle \chi_{i}, g_r \rangle \, \langle \chi_{i^*}, g_s \rangle.
  \end{equation*}
  Now~\eqref{eq:ct-2nd-ortho} follows from Lemma~\ref{lem:ct-0-th-row}. To prove \eqref{eq:ct-1st-ortho}, we consider the matrices $T$ and $T'$ whose $(i, j)$-entries are given by
  \begin{equation*}
    T_{i j} = \langle \chi_{i^*}, g_j \rangle
    \quad \text{and} \quad
    T'_{i j} = \frac{|\mathfrak{C}_i|}{\dim(\mathcal{C})} \langle \chi_j, g_i \rangle
    \quad (i, j = 0, \dotsc, m),
  \end{equation*}
  respectively. Equation \eqref{eq:ct-2nd-ortho} means $T' \cdot T = 1$.
  Hence, by linear algebra, we also have $T \cdot T' = 1$. This implies \eqref{eq:ct-1st-ortho}. Equation \eqref{eq:ct-chg-basis-2} is easily verified by using Theorem~\ref{thm;ct-idempo} and the first orthogonality relation.
\end{proof}

Now we suppose that $\mathrm{char}(k) = 0$.

\begin{theorem}
  Every entry of the character table of $\mathcal{C}$ is a cyclotomic integer.
\end{theorem}

\begin{proof}
  Let $i, r \in \{ 0, \dotsc, m \}$. Since $f_r \in \CF(\mathcal{C})$ is a primitive idempotent, there is an algebra map $\rho_r: \CF(\mathcal{C}) \to k$ such that $f \star f_r = \rho_r(f) f_r$ for all $f \in \CF(\mathcal{C})$. By \cite[Corollary 8.53]{MR2183279}, $\rho_r(\chi_i)$ is a cyclotomic integer. By~\eqref{eq:ct-chg-basis-2}, we have
  \begin{equation*}
    \rho_r(\chi_i) f_r = \chi_i \star f_r
    = \sum_{s = 0}^m \langle \chi_{i}, g_s \rangle \, f_s \star f_r
    = \sum_{s = 0}^m \langle \chi_{i}, g_s \rangle \Kdelta_{r,s} f_r
    = \langle \chi_i, g_r \rangle \, f_r,
  \end{equation*}
  and hence $\rho_r(\chi_i) = \langle \chi_i, g_r \rangle$. Therefore the $(i,r)$-th entry of the character table, $\langle \chi_i, g_r \rangle$, is a cyclotomic integer.
\end{proof}

\begin{example}
  \label{ex:grp-char-table}
  We consider the case where $\mathcal{C} = kG\mbox{-{\sf mod}}$ is the category of representations of a finite group $G$ such that $|G| \ne 0$ in $k$.
  As in \S\ref{subsec:Hopf-alg} and \S\ref{subsec:Hopf-alg-Fourier}, we identify $\mathcal{Z}(\mathcal{C})$ with the Yetter-Drinfeld category over $k G$.
  Then $R(\unitobj)$ is the vector space $k G$ with the action $\triangleright$ and the coaction given by
  \begin{equation*}
    g \triangleright a = g a g^{-1}
    \quad \text{and} \quad
    a \mapsto a \otimes a
    \quad (a, g \in G),
  \end{equation*}
  respectively. Now let $\mathfrak{C}_0, \dotsc, \mathfrak{C}_m$ be the conjugacy classes of $G$ with $1 \in \mathfrak{C}_0$.
  Then we have the following decomposition in the Yetter-Drinfeld category:
  \begin{equation*}
    R(\unitobj) = \bigoplus_{r = 0}^m \, \mathrm{span}_k \, \mathfrak{C}_r.
  \end{equation*}
  Thus we can (and do) identify a conjugacy class in the sense of Definition~\ref{def:conj-class} with a conjugacy class of $G$ in the usual sense.
  Since the pivotal dimension coincides with the ordinary dimension in our case, the size of $\mathfrak{C}_r$ in the sense of Definition~\ref{def:conj-class} is precisely the number of elements of the conjugacy class $\mathfrak{C}_r$.

  Let $\lambda: k G \to k$ be the cointegral such that $\lambda(1) = 1$, and let $\Lambda \in k G$ be the integral such that $\langle \lambda, \Lambda \rangle = 1$.
  Explicitly, $\lambda(g) = \Kdelta_{g, 1}$ ($g \in G$) and $\Lambda$ is the sum of all elements of $G$ in $k G$.
  For $r = 0, \dotsc, m$, we define $f_r$ and $\widetilde{f}_r$ as above.
  Then the class function $f_r$ is the characteristic function on the subset $\mathfrak{C}_r \subset G$.
  Thus, by \eqref{eq:Fourier-Hopf-inv}, the class sum associated to $\mathfrak{C}_r$ is given by
  \begin{equation*}
    \mathfrak{F}_{\lambda}^{-1}(f_{r})
    = \sum_{g \in G} \langle f_r, g \rangle \, g = \sum_{g \in \mathfrak{C}_r} g.
  \end{equation*}
  Hence the element $g_r$ defined by~\eqref{eq:conj-class-aver} is given by $g_r = |\mathfrak{C}_g|^{-1} \sum_{g \in \mathfrak{C}_r} g$.
  From this, we see that the character table of $\mathcal{C}$ in the sense of Definition~\ref{def:ct} agrees with the character table of $G$ in the usual sense.

  The above understanding of conjugacy classes by the Yetter-Drinfeld category is due to Zhu \cite{MR1402892}.
  Cohen and Westreich \cite{MR2863455} have defined a conjugacy class of a semisimple Hopf algebra in the spirit of Zhu.
  Our definitions basically agree with Cohen and Westreich's (see also \S\ref{subsec:Hopf-alg-Fourier}).
  It is interesting to investigate how can we generalize the results of \cite{MR2863455,MR3195413} to the setting of fusion categories.
\end{example}

\begin{example}[Modular tensor categories]
  \label{ex:ch-t-MTC}
  Suppose that the fusion category $\mathcal{C}$ is a ribbon category with braiding $\sigma$. The $S$-matrix $S = (s_{i j})$ of $\mathcal{C}$ is then defined by
  \begin{equation*}
    s_{i j} = \trace(\sigma_{V_j,V_i^*} \sigma_{V_i^*,V_j})
    \quad (i, j = 0, \dotsc, m).
  \end{equation*}
  Set $d_j = s_{0 j}$. By \cite[Theorem 3.1.12]{MR1797619} and its proof, the linear map
  \begin{equation*}
    Q: \CF(\mathcal{C}) \to \CE(\mathcal{C}),
    \quad Q(\chi_i) = \sum_{j = 0}^m \frac{s_{i j}}{d_j} e_j
    \quad (i = 0, \dotsc, m)
  \end{equation*}
  is in fact a homomorphism of algebras.

  A {\em modular tensor category} \cite{MR1797619} is a ribbon fusion category whose $S$-matrix is invertible.
  If $\mathcal{C}$ is a modular tensor category, then $Q$ is an isomorphism of algebras.
  After relabeling indices, we may assume that $Q(f_j) = e_j$ for all $j$. Then
  \begin{equation*}
    \chi_i = Q^{-1} Q(\chi_i)
    = \sum_{j = 0}^m \frac{s_{i j}}{d_j} Q^{-1}(e_j)
    = \sum_{j = 0}^m \frac{s_{i j}}{d_j} f_j
  \end{equation*}
  for all $i$. It follows from Corollary~\ref{cor:ct-ortho} that the matrix $(s_{i j} /  d_j)$ is the character table of $\mathcal{C}$.
  This result generalizes \cite[Theorem 3.2]{MR2863455}.
\end{example}


\def\cprime{$'$}
\providecommand{\bysame}{\leavevmode\hbox to3em{\hrulefill}\thinspace}
\providecommand{\MR}{\relax\ifhmode\unskip\space\fi MR }
\providecommand{\MRhref}[2]{%
  \href{http://www.ams.org/mathscinet-getitem?mr=#1}{#2}
}
\providecommand{\href}[2]{#2}

\end{document}